\documentclass{amsart}
\numberwithin{equation}{section}
\usepackage[utf8]{inputenc}
\usepackage[T1]{fontenc}
\usepackage{lmodern}
\usepackage[margin=1.2in]{geometry}
\usepackage{setspace}
\usepackage[usenames, dvipsnames]{xcolor}
\usepackage{graphicx}
\usepackage{mathtools}
\usepackage{amssymb}
\usepackage{amsthm}
\usepackage{xspace}
\usepackage{tikz-cd}
\usepackage{slashed}
\usepackage[backref=page, bookmarks=false]{hyperref}
\usepackage[capitalize, noabbrev]{cleveref}
\usepackage{braket}

\tikzcdset{scale cd/.style={every label/.append style={scale=#1},
    cells={nodes={scale=#1}}}}

\newcommand{\mcal}[1]{\mathcal{#1}}
\newcommand{\cF}{\mcal{F}}
\newcommand{\op}{\mathsf{op}}

\newsavebox{\pullback}
\sbox\pullback{%
\begin{tikzpicture}%
\draw (0,0) -- (1ex,0ex);%
\draw (1ex,0ex) -- (1ex,1ex);%
\end{tikzpicture}}

\DeclareMathOperator{\Hom}{Hom}

\newtheorem{lemma}[equation]{Lemma}
\newtheorem{corollary}[equation]{Corollary}
\newtheorem{proposition}[equation]{Proposition}

\newtheorem{theorem}[equation]{Theorem}

\newtheorem*{lemma*}{Lemma}

\theoremstyle{definition}
\newtheorem{example}[equation]{Example}
\newtheorem{definition}[equation]{Definition}

\newtheorem{observation}[equation]{Observation}
\newtheorem{notation}[equation]{Notation}

\theoremstyle{remark}
\newtheorem{remark}[equation]{Remark}

\crefname{thm}{Theorem}{Theorems}
\crefname{lem}{Lemma}{Lemmas}
\crefname{cor}{Corollary}{Corollaries}
\crefname{prop}{Proposition}{Propositions}
\crefname{ex}{Exercise}{Exercises}
\crefname{exm}{Example}{Examples}
\crefname{defn}{Definition}{Definitions}
\crefname{claim}{Claim}{Claims}
\crefname{rem}{Remark}{Remarks}
\crefname{fct}{Fact}{Facts}
\crefname{note}{Note}{Notes}

\makeatletter
	\let\oldparen\paren
	\def\paren{\@ifstar{\oldparen}{\oldparen*}}
\makeatother

\usepackage{microtype}
\usepackage{hypcap}

\hypersetup{
 colorlinks,
 linkcolor={red!50!black},
 citecolor={green!50!black},
 urlcolor={blue!80!black}
}

\newcommand{\cat}{\mathrm{Cat}_{\infty}}

\newcommand{\id}{\mathrm{id}}

\newcommand{\Sp}{\mathrm{Sp}}
\newcommand{\End}{\mathrm{End}}

\newcommand{\cL}{\mathcal{L}}

\newcommand{\mbbC}{\mathbb{C}}

\newcommand{\mbbE}{\mathbb{E}}

\newcommand{\mbbK}{\mathbb{K}}

\newcommand{\mbbN}{\mathbb{N}}

\newcommand{\mbbZ}{\mathbb{Z}}

\newcommand{\Mod}{\mathrm{Mod}}

\newcommand{\BMod}{\mathrm{BMod}}

\newcommand{\SBim}{\mathrm{SBim}}

\newcommand{\EE}{\mathbb{E}}
\newcommand{\U}{\mathrm{U}}
\newcommand{\BS}{\mathrm{BS}}

\newcommand{\pt}{\mathrm{pt}}

\renewcommand{\Set}{\mathrm{Set}}

\newcommand{\eHom}{\underline{\Hom}}

\newcommand{\K}{{\bK}}

\newcommand{\Kb}{\K^b}
\newcommand{\Kbloc}{\Kb_{\mathrm{loc}}}

\newcommand{\Cat}{\mathrm{Cat}}

\newcommand{\Mor}{\mathrm{Mor}}

\newcommand{\Morita}{\mathrm{Morita}}

\newcommand{\Moritac}{\Morita^{\mrc}}

\def\cA{\mathcal A}\def\cB{\mathcal B}\def\cC{\mathcal C}\def\cD{\mathcal D}
\def\cF{\mathcal F}
\def\cL{\mathcal L}

\def\cZ{\mathcal Z}

\def\EE{\mathbb E}

\newcommand{\h}{{\sf h}}

\makeatletter
\providecommand{\leftsquigarrow}{%
  \mathrel{\mathpalette\reflect@squig\relax}%
}
\newcommand{\reflect@squig}[2]{%
  \reflectbox{$\m@th#1\rightsquigarrow$}%
}
\makeatother

\newcommand{\Alg}{\mathrm{Alg}} 
 \newcommand{\Braid}{\mathrm{Braid}}

\newcommand{\loc}{\mathrm{loc}}

\newcommand{\CProj}{\mathrm{CProj}}

\newcommand{\add}{\mathrm{add}}

\newcommand{\st}{\mathrm{st}}

\newcommand{\Monoid}{\cZ}

\newcommand{\Spectra}{\mathrm{Sp}}
\newcommand{\ConnSpectra}{\Spectra_{\geq 0}}
\newcommand{\largecat}{\widehat{\mathrm{Cat}}}

\newcommand{\Perf}{\mathrm{Perf}}

\newcommand{\CAlg}{\mathrm{CAlg}}
\newcommand{\mrc}{\mathrm{c}}
\newcommand{\cp}{\mathrm{cp}}

\newcommand{\arrow}{\mathrm{ar}}

\newcommand{\PreBraid}{\mathrm{PreBraid}}

\newcommand{\Gen}{\mathrm{Gen}}

\renewcommand{\th}{{\text{th}}}

\newcommand{\Free}{\mathrm{Free}}

\def\bK{\mathbf K}

\newcommand{\EBS}{EB}
\newcommand{\FE}{{\mcal{F}_{E}}}

\newcommand{\MU}{\mathrm{MU}}
\newcommand{\PolyMorc}{\mathrm{Poly}\Morc}
\newcommand{\Morc}{\Mor^{\mrc}}
\newcommand{\fib}{\mathrm{fib}}
\newcommand{\underbf}[1]{\underline{\mathbf{{#1}}}}
\newcommand{\Deltas}{\Delta s}

\title{Braiding on complex oriented Soergel bimodules}
\address{1 Oxford St, Cambridge, MA 02139}
\email{yuleonliu@math.harvard.edu}
\author{Yu Leon Liu}
\begin{document}
\begin{abstract}
    In this note, we study $\U(n)$ Soergel bimodules in the context of stable homotopy theory. We define the $(\infty, 1)$-category $\SBim_E(n)$ of $E$-valued $\U(n)$ Soergel bimodules, where $E$ is a connective $\EE_\infty$-ring spectrum, and assemble them into a monoidal locally additive $(\infty, 2)$-category $\SBim_E$. When $E$ has a complex orientation, we then construct a braiding, i.e. an $\EE_2$-algebra structure, on the universal locally stable $(\infty, 2)$-category $\Kbloc(\SBim_E)$ associated to $\SBim_E$. Along the way, we also prove spectral analogs of standard splittings of Soergel bimodules. This is a topological generalization of the type $A$ Soergel bimodule theory developed in \cite{2024braided}. 
\end{abstract}
\maketitle
\tableofcontents

\section{Introduction}
In the last twenty years, there has been tremendous progress in quantum topology and categorification, with 
\emph{Soergel bimodules} \cite{Soergel} at the center of much of the advancement (see \cite{stroppel2022categorification}).
While Soergel bimodules can defined for any compact Lie group $G$, in this paper we restrict ourselves to the $\U(n)$ case and denote by $\SBim(n)$ the category of $\U(n)$ Soergel bimodules.
The $\U(n)$ Soergel bimodules can be packaged together into a $(2,2)$-category\footnote{We invite the reader to \cite[Appendix A]{2024braided} for an accessible introduction to higher and $\infty$-categories, which are utilized throughout this paper.}  $\SBim$, 
whose objects are indexed by natural numbers $n \in \mbbN$, and whose endomorphism category $\eHom_{\SBim}(n,n)$ is $\SBim(n)$.
In \cite{2024braided} the author and collaborators constructed a braiding, i.e. an $\EE_2$-algebra structure, on the locally stable $(\infty, 2)$-category $\Kbloc(\SBim)$, which is the universal locally stable $(\infty, 2)$-category associated to $\SBim$.\footnote{The objects of $\Kbloc(\SBim)$ are also indexed by natural numbers, and the endomorphism category $\eHom_{\Kbloc(\SBim)}(n,n) = \Kb(\SBim(n))$ is the universal stable $\infty$-category associated to $\SBim(n)$, which is the $\infty$-categorical generalization of the chain homotopy category. See \cref{rem:Kb-generalizes-chain-complex}.} Just as the $S_n$ Hecke algebras taken together form a braided monoidal category that controls type $A$ quantum link invariants and $3$d topological field theories, we believe that $\Kbloc(\SBim)$, together with its $\EE_2$-algebra structure, controls type $A$ homotopy-coherent link homology theories and $4$d topological quantum field theories (see the introduction of \cite{2024braided}).

In this paper we generalize the results of \cite{2024braided} to the setting of stable homotopy theory.\footnote{See \cite{nituSBI, nituSBII} for previous work on spectral generalizations of Soergel bimodules.}
Let $T$ denote the maximal torus $\U(1)^n$ in $\U(n)$. The 
$\U(n)$ Soergel bimodules are retracts of direct sums of Bott-Samelson bimodules, which are $T-T$-equivariant rational cohomology groups of Bott-Samelson spaces (\cref{def:BS-space}). 
Let $E$ be an $\EE_\infty$-ring spectrum. We can similarly define $E$-valued Soergel bimodules as retracts of direct sums of $E$-valued Bott-Samelson bimodules, which are $E$-mapping spectra of  $T-T$-equivariant quotients of Bott-Samelson spaces (see \cref{def:E-Bott-Samelson}). Let $\SBim_E(n)$ denote the full subcategory of $E(BT)-E(BT)$-bimodule spectra consisting of $E$-valued Soergel bimodules. 
Our first result is that $\SBim_E(n)$ can be assembled into a monoidal locally additive $(\infty, 2)$-category:
\begin{theorem}[{\cref{prop:spell-out-SBim_E}}]
    Let $E$ be a connective $\EE_\infty$-ring spectrum.\footnote{We worked with connective ring spectra for technical reasons. However, the connectivity condition can be removed; see \cref{rem:non-connective-E}.} There exists a monoidal locally additive $(\infty, 2)$-category $\SBim_E$ whose objects are labeled by natural numbers, and whose hom categories are 
    \begin{equation}
        \eHom_{\SBim_E}(n,m) = 
        \begin{cases}
            0 & n \neq m, \\
            \SBim_E(n) & n = m.
        \end{cases}
    \end{equation}
\end{theorem}
By abstract nonsense, we also get a monoidal locally stable $(\infty, 2)$-category $\Kbloc(\SBim_E)$ (see \cref{obs:describe-KblocSBim_E}) whose objects are also labeled by natural numbers, and whose hom categories are 
\begin{equation}
    \eHom_{\Kbloc(\SBim_E)}(n,m) = 
    \begin{cases}
        0 & n \neq m, \\
        \Kb(\SBim_E(n)) & n = m. 
    \end{cases}
\end{equation}
Here $\Kb(\SBim_E(n))$ is the universal stable $(\infty, 1)$-category associated to $\SBim_E(n)$.
Additionally, we have a \emph{fiber functor} \eqref{eq:fiber-functor}, which 
is an $\EE_1$-algebra map $H_{\loc} \colon \Kbloc(\SBim_E) \to \Mor_E$, where $\Mor_E$ is the symmetric monoidal $(\infty, 2)$-category of $E$-algebras and bimodules (satisfying some finiteness condition), and bimodule homomorphisms.

Next we turn to the braiding. Interestingly, we need a complex $\EE_\infty$-orientation on $E$, which is an $\EE_\infty$-ring homomorphism $f_E \colon \MU \to E$, where $\MU$ is the complex bordism spectrum.
Intuitively, such an orientation is necessary because it allows us to trivialize Thom spectra of complex vector bundles (see \cref{prop:E-thom-class} and \cref{prop:E-two-fiber-seq}).

Here is the main result of the paper:
\begin{theorem}[\cref{thm:main}]
    Let $E$ be a connective $\EE_\infty$-ring spectrum with a complex $\EE_\infty$-orientation $f_E \colon \MU \to E$. There exists an $\EE_2$-algebra structure on $\Kbloc(\SBim_E)$ together with an $\EE_2$-algebra structure on $H_{\loc} \colon \Kbloc(\SBim_E) \to \Mor_E$ such that 
    \begin{enumerate}
        \item The $\EE_2$-algebra structures enhances the $\EE_1$-algebra structures on $\Kbloc(\SBim_E)$ and $H_{\loc}$.
        \item The braiding on two strands is given by the Rouquier complex (\cref{def:rouquier-complex}).
    \end{enumerate}
   Furthermore, the space of such pairs of $\EE_2$-algebra structures satisfying the above conditions is contractible.
\end{theorem}
\begin{remark}
   See \cref{thm:main} for the precise statement.
\end{remark}
\begin{remark}
    A crucial step in the proof of \cref{thm:main} is constructing the splittings $\BS_{\underline{ii}} \simeq \BS_{i} \oplus \BS_{i}(2)$ (\cref{prop:E-BS_s-splitting-1}) and $\BS_{\underline{iji}} \simeq \BS_{i,j} \oplus \BS_{i}(2)$ (\cref{prop:E-sts-splitting-1}) in the setting of $E$-valued Soergel bimodules. Here $j = i \pm 1$ and $\BS_{i,j}$ is defined in \cref{nota:EBS_iji}.
\end{remark}
\begin{remark}
    Let $E = H\mbbZ$ together with its canonical complex $\EE_\infty$-orientation (see \cref{ex:HS-additive}). 
    It follows immediately from \cref{thm:main} that the braiding constructed in \cite{2024braided} on rational Soergel bimodules can be lifted to integral Soergel bimodules.
\end{remark}





\textbf{Outline.}
In \cref{sec:top-of-BS} we introduce Bott-Samelson spaces and construct two important cofiber sequences.
In \cref{sec:complex-oriented-BS} we study  $E$-valued Bott-Samelson bimodules and their splittings, where $E$ is equipped with a complex $\EE_\infty$-orientation. 
In \cref{sec:inf-2-SBim} we construct the $(\infty, 2)$-categories of additive Soergel categories $\SBim_E$, as well as the stable analogue $\Kbloc(\SBim_E)$.
Lastly, in \cref{sec:braiding-on-Kbloc} we prove  \cref{thm:main} regarding braidings on $\Kbloc(\SBim_E)$.

\textbf{Acknowledgements:}
I would like to thank Cameron Krulewski and Sophia Se\~na for helpful comments on an early draft of this paper.
I would like to thank Aaron Mazel-Gee, David Reutter, Catharina Stroppel, and Paul Wedrich for previous collaboration, which both inspired and formed the technical foundation of this work. I would also like to thank Fabio Capovilla-Searle, Alex Hsu, Catherine Li,  Ryan Quinn, Alba Send\'{o}n Blanco, and Daniel Tolosa for many pleasant mathematical and non-mathematical conversations. Lastly, I extend my gratitude to Mike Hopkins and Tomer Schlank for insightful discussion and encouragement.
A significant part of the project was completed while I attended the 2024 Algebraic Structures in Topology conference held in San Juan, Puerto Rico.
I gratefully acknowledge the financial support provided by the Simons Collaboration on Global Categorical Symmetries.

\section{Topology of Bott-Samelson spaces}\label{sec:top-of-BS}
In this section, we review the basics of Bott-Samelson spaces. We will restrict ourselves to $\U(n)$, although much of the following works for general compact Lie groups.

\subsection{The basics of Bott-Samelson spaces}
\label{subsec:basics-of-BS}
Fix $n > 1$. Let $G$ be $\U(n)$, $T$ be $\U(1)^n$ the maximal torus, and $W = N_G(T)/T = S_n$ be the Weyl group. 
Given a simple transposition $s_i = (i, i+1)$, we denote by $G_{i}$ the corresponding standard parabolic  $\U(1)^{i-1} \times \U(2) \times \U(1)^{n-i-1}$. Note that $G_i$'s, along with other standard parabolic subgroups, are spaces with a
$T-T$ action by $T$ multiplication on the both the left and the right.\footnote{A parabolic subgroup is \emph{standard} if it contains $T$. They are in one-to-one correspondence with parabolic subgroups of $W$.  
}



Given two $T-T$ spaces $X_1, X_2$, we can define a new $T-T$ space $X_1 \times_T X_2$ as the homotopy quotient by the simultaneous $T$ action, acting on the right of $X_1$ and the left of $X_2$.\footnote{In this paper the left and right $T$ actions are free (but the combined $T \times T$ action might not be!), therefore the naive quotient and the homotopy quotient coincide.} The $T-T$ action is defined to be 
\begin{equation}
    (t_1, t_2) \cdot (x_1, x_2) \coloneqq (t_1 \, x_1, x_2\, t_2).
\end{equation}
We will view the operation $- \times_T -$ as composition, as it will correspond to composition of $1$-morphisms.
We now define Bott-Samelson spaces as spaces with $T-T$ actions:
\begin{definition}\label{def:BS-space}
    Fix  $\underbf{i} = (i_1, \cdots, i_m)$ in $\{1, \cdots, n-1\}^{\times m}$ with $m \in \mbbN$. The \emph{Bott-Samelson} space $\BS_{\underbf{i}}$ is the $T-T$ space 
    \begin{equation}
        G_{i_1} \times_T G_{i_2} \cdots \times_T G_{i_m}.
    \end{equation}
    When $\mathbf{i} = \varnothing$ we take $\BS_{\varnothing}$ to be $T$.
\end{definition}

\begin{remark}
    $T-T$ acts on the left of $G_{i_1}$ and right of $G_{i_m}$ respectively:
    \begin{equation}
        (t_1, t_2) \cdot (g_1, \cdots, g_m) \coloneqq (t_1\, g_1, \cdots, g_m\, t_2).
    \end{equation}
\end{remark}

\begin{remark}\label{rem:Bott-Samelson-variety}
    The more standard notion of Bott-Samelson {variety} \cite{bottsamelson} associated to $\underbf{i}$ is the quotient $\BS_{\underbf{i}}/T$. 
\end{remark}

We end the subsection with a useful proposition needed for the slide maps. 
Let $s_i, s_j$ be neighboring simple transpositions, i.e. $j = i\pm 1$. 

The elements $s_i$ and $s_j$ generate a $S_3$ subgroup in $W$. We denote by $G_{i,j} = G_{j,i}$
the associated standard parabolic subgroup of $G$.\footnote{If $j = i+1$, then $G_{i,j} = 
\U(1)^{i-1} \times \U(3) \times \U(1)^{n-i-2}$.} Then we have a commutative diagram of $T-T$ spaces:
\begin{equation}\label{eq:sts-pushout}
    \begin{tikzcd}
        \BS_{\underline{ii}} = G_i \times_T G_i \ar[rrrr,
            "{(g_1, g_2) \mapsto (g_1, 1, g_2)}"
            ] \ar[d, "\mu_i"] 
            &&&&
            \BS_{\underline{iji}} = G_i \times_T G_{j} \times_T G_i
            \ar[d, "\mu_{iji}"]\\ 
        \BS_s = G_i \ar[rrrr] &&&&
        G_{i,j}
    \end{tikzcd}
\end{equation}
where the vertical maps are multiplications. We have the following propositions from \cite[Lemma 5.7]{nituSBI} (with $K = \{s_i\}$):
\begin{proposition}\label{prop:sts-pushout}
   The diagram \eqref{eq:sts-pushout} is a pushout diagram.
\end{proposition}

\subsection{Two fiber sequences}\label{subsec:two-fiber-sequences}

In this subsection we construct two fiber sequences that are crucial in constructing the braiding. 
To start, fix a simple transposition $s_i \in W = N_G(T)/T$. 
Pick a lift to $N_G(T)/T$, for which we abuse notation and also denote by $s_i$.\footnote{One such lift is the permutation matrix associated to $s$.}
Conjugation by $s_i$ on $T$ induces the Weyl group element's action on $T$: $$\sigma_i \colon T \to T, \quad t\mapsto \sigma_i(t) \coloneqq s_i\, t\, s_i^{-1}.$$
\begin{definition}
    Let $Ts_i$ be the $T-T$ space whose underlying space is $T$ and the $T-T$ action is given by 
    \begin{equation}
        (t_1, t_2) \cdot t \coloneqq t_1\, t \, (\sigma_i(t_2)) = t_1\, t\, (s_i\, t_2\, s_i^{-1}).
    \end{equation}
\end{definition}
Note that we conjugated the right $T$ action by $s_i$. 
\begin{definition}\label{def:m-and-ms}
Let $m_i \colon T \to G_i$ denote the standard $T-T$-equivariant inclusion. Similarly, let $ms_i \colon Ts_i \to G_i$ denote the map $t \mapsto t\,s$. It is straightforward to check that $ms_i$ is also $T-T$-equivariant.
\end{definition}
Let $H$ be a group and $X$ be a space with $H$ action. Furthermore, let $V$ be an $H$-equivariant vector bundle on $X$; then the Thom space $X^V$ is a $H$-equivariant based space. 
In our case, let $\Lambda_T \coloneqq \Hom(T, U(1))$ be the character lattice. Each character $\beta \in \Lambda_T$ defines a one-dimensional complex representation of $T$, i.e. a $T$-equivariant line bundle over the trivial $T$-space $* = T/T$. Equivalently, we get a $T-T$-equivariant line bundle $\cL_\beta$ on $T$. 
Applying the same argument to $Ts_i$, $\beta$ also defines a $T-T$-equivariant line bundle $\cL_\beta$ on $Ts_i$.
\begin{proposition}\label{prop:two-fiber-sequences}
    Let $\alpha_i$ be the root associated to the simple transposition $s_i$.\footnote{$\alpha_i$ has $1$ in the $i$-th row, $-1$ in the $(i+1)$-th row, and $0$ everywhere else.}
    As a $T-T$ based space, the cofiber of $ms_i$ is equivalent to the Thom space $T^{\cL_{\alpha_i}}$. Furthermore, let $\Delta_i$ denote the cofiber map $G_i \to T^{\cL_\alpha}$; then the composite 
    \begin{equation}
        T \xrightarrow{m_i} G_i \xrightarrow{\Delta_i} T^{\cL_{\alpha_i}}
    \end{equation}
    is the zero section map.

    Similarly, the cofiber of $m_i$ is equivalent to the Thom space $(Ts_i)^{\cL_{-\alpha_i}}$. Furthermore, let $\Deltas_i$ denote the cofiber map $G_i \to Ts_i^{\cL_{-\alpha_i}}$; then the composite 
    \begin{equation}
        Ts_i \xrightarrow{ms_i} G_i \xrightarrow{\Deltas_i} (Ts_i)^{\cL_{-\alpha_i}}
    \end{equation}
    is the zero section map.
    To summarize, we get a diagram:
\begin{equation}\label{eq:fancy-diagram}
    \begin{tikzcd}
        T \ar[rd, "m_i"]& & Ts_i \ar[ld, "ms_i"'] \\ 
        & G_i  \ar[ld, "\Delta_i"'] \ar[rd, "\Deltas_i"]& \\
        T^{\cL_{\alpha_i}} & & (Ts_i)^{\cL_{-\alpha_i}}
    \end{tikzcd}
\end{equation}
where the diagonal maps are cofiber sequences while the other two composites are zero section inclusions.
\end{proposition}
\begin{proof}
    We claim that 
    there are $T-T$-equivariant tubular neighborhoods $\nu_0$, $\nu_\infty$ of $T$ and $Ts$ (via inclusions $m_i$ and $ms_i$) in $G_i$ such that $v_0 \bigcup v_\infty = G_i$ and $v_0 \bigcap v_\infty = \partial \nu_0 = \partial \nu_\infty$.

    To construct them, we first quotient $G_i$ by the right $T$ action. Note that $G_i/T \simeq \mbbC P^1 = (\mbbC^2-\{0\})/\mbbC^\times$ with $T$ acting as
    \begin{equation}\label{eq:T-action-on-CP1}
        t \cdot [x_1 : x_2] \coloneqq [x_1 : \alpha_i(t)\, x_2] = [\alpha_i(t)^{-1} \,x_1 : x_2] = [(-\alpha_i)(t)\, x_1: x_2]
    \end{equation}
     Furthermore, the maps $m_i/T\colon T/T = \pt \to \mbbC P^1$ 
     and $ms_i/T \colon Ts_i/T = \pt \to \mbbC P^1$
     correspond to the inclusions of $T$-fixed points $[1 : 0]$ and $[0:1]$ respectively. We can  identify $\mbbC P^1$ with $S^2$ such that the two $T$-fixed points corresponds to the north and south poles. Furthermore, the northern and southern semispheres form $T$-equivariant tubular neighborhoods of the north and south poles. 
    Let $\nu_0,\, \nu_\infty$ be the preimages of the northern and southern semispheres along the quotient map $G_i \to G_i/T \simeq S^2$. It is clear that they are tubular neighborhoods of $T$ and $Ts_i$ satisfying the conditions above.

    It follows that the cofiber of $ms_i$ is $T/v_\infty \simeq v_0/\partial v_0 = T^{V}$ where $V$ is the normal bundle of $T$ in $G_i$. Now we need to identify $V$ with $\cL_{\alpha_i}$. Once again we quotient out the free right $T$ action; then $V$ corresponds to the $T$ representation that is the tangent space of the $T$-fixed point $[1:0]$ in $\mbbC P^1$. The tangent space at $[1:0]$ can be identified with the coordinate $x_2$, on which $T$ acts via the character $\alpha_i$ by \eqref{eq:T-action-on-CP1}. Hence $V$ is the line bundle $\cL_\alpha$.
    The same argument shows that the cofiber of $m_i$ is $(Ts_i)^{\cL_{-\alpha_i}}$. Note that the $-\alpha_i$ comes from the action of $T$ on the coordinate $x_1$ in \cref{eq:T-action-on-CP1}.
\end{proof}

\begin{remark}\label{rem:heegaard-splitting}
    \cref{prop:two-fiber-sequences} generalizes to general compact Lie groups. In particular, let $G = SU(2) \simeq S^3$, $T = \U(1)$, $W = S_n$, and $s_1 = (1,2)$. Then we can view $T$ and $Ts_i$ as two unknots that link each other in $S^3$. Furthermore, the decomposition of $S^3$ into two tubular neighborhood is precisely the standard genus one Heegaard splitting of $S^3$.
\end{remark}

Next we compose the two fiber sequences with $G_i$. 
Note that we have obvious $T-T$-equivariant isomorphisms:
\begin{equation}\label{eq:obvious-iso}
    \begin{tikzcd}
        G_i \times_T T \ar[r, "\sim"] &  G_i & \ar[l, "\sim"'] T \times_T G_i \\ 
        (g,t) \ar[r, mapsto] & g\,t & \\
        & t\,g & \ar[l, mapsto] (t, g).
    \end{tikzcd}
\end{equation}

As for $Ts_i$, we use the fact that $s_i$ is an element of $G_i$ to get isomorphisms:
\begin{equation}\label{eq:non-obvious-iso}
   \begin{tikzcd}
    G_i \times_T Ts_i \ar[r, "\sim"] &  G_i & \ar[l, "\sim"'] Ts_i \times_T G_i \\ 
    (g,t) \ar[r, mapsto] & g\,t\,s_i & \\
    & t\,s_i\,g & \ar[l, mapsto] (t, g).
   \end{tikzcd} 
\end{equation}
It is  straightforward to check that these are well-defined and are $T-T$-equivariant.

Fix two $T-T$ spaces $X_1, X_2$. Suppose that we have a $T-T$-equivariant vector bundle $V$ on $X_1$. Then we get a $T-T$-equivariant vector bundle $V \times_E X_2$ over $X_1 \times_T X_2$ whose total space is $E_V \times_T X_2$, where $E_V$ is the total space of $V$. 
\begin{corollary}\label{cor:four-fiber-sequences}
By applying $-\times_T G_i$ to \eqref{eq:fancy-diagram}, we get a diagram 
\begin{equation}\label{eq:fancy-diagram-times-Gs-1}
    \begin{tikzcd}
        G_i \ar[rd, "m_i \times_T G_i"]& & G_i \ar[ld, "ms_i \times_T G_i"'] \\ 
        & G_i \times_T G_i  \ar[ld, "\Delta_i \times_T G_i"'] \ar[rd, "\Deltas_i \times_T G_i"]& \\
        G_i^{\cL_\alpha \times_T G_i} & & (G_i)^{\cL_{-\alpha}' \times_T G_i}
    \end{tikzcd}
\end{equation}
where the diagonal maps are cofiber sequences and the other two composites are zero section inclusions.
Note that we twisted $ms_i \times_T G_i$ and $\Deltas_i \times_T G_i$ by the isomorphisms in \eqref{eq:non-obvious-iso}.

Similarly, by applying $G_i \times_T  -$ to \eqref{eq:fancy-diagram}, we get a diagram 
\begin{equation}\label{eq:fancy-diagram-times-Gs-2}
    \begin{tikzcd}
        G_i \ar[rd, "G_i \times_T  m_i "]& & G_i \ar[ld, "G_i \times_T  ms_i "'] \\ 
        & G_i \times_T G_i  \ar[ld, "G_i \times_T  \Delta_i"'] \ar[rd, "G_i \times_T  \Deltas_i"]& \\
        G_i^{G_i \times_T \cL_\alpha} & & (G_i)^{G_i \times_T \cL_{-\alpha}'}
    \end{tikzcd}
\end{equation}
where the diagonal maps are cofiber sequences and the other two composites are zero section inclusions.
Once again we twisted $G_i  \times_T ms_i$ and $G_i \times_T \Deltas_i$ by the isomorphisms in \eqref{eq:non-obvious-iso}.
\end{corollary}
\begin{observation}\label{obs:explicit-maps-Gi}
    Unpacking the construction, 
    $m_i \times_T G_i$ and 
    $ms_i \times_T G_i$ are given by 
    $g \mapsto (1, g)$ and 
    $g \mapsto (s_i, s_i^{-1}\,g)$ respectively.
    Similarly, 
    $G_i \times_T m_i$ and 
    $G_i \times_T ms_i$ are given by 
    $g \mapsto (g,1)$ and
    $g \mapsto (gs_i^{-1}, \, s_i)$ respectively.
\end{observation}
Let us end this section with an easy yet important observation:
\begin{observation}\label{obs:mult-splitting}
    \cref{obs:explicit-maps-Gi} implies that the multiplication map $$\mu_{i} \colon G_i \times_T G_i \to G_i, \quad  (g_1, g_2) \mapsto g_1\, g_2$$ is a $T-T$-equivariant retraction of all four maps 
    $m_i \times_T G_i$, 
    $ms_i \times_T G_i$,
    $G_i \times_T m_i$, and
    $G_i \times_T ms_i$.
\end{observation}



\section{Complex oriented Bott-Samelson bimodules}\label{sec:complex-oriented-BS}
In this section we define Bott-Samelson bimodules valued in complex $\EE_\infty$-oriented ring spectra and study their splittings.
We start by reviewing the basics of formal group laws.

\subsection{Basics of formal group laws}\label{subsec:basics-of-FGL}
Let $S$ be a commutative ring. 
Recall that a (one-dimensional) formal group law $\cF$ on  $S$ is a power series $\cF \in S[[x, y]]$ satisfying
\begin{enumerate}
    \item $\cF(x,y) = \cF(y,x)$,
    \item $\cF(0,x) = \cF(x,0)$,
    \item $\cF(x, \cF(y,z)) = \cF(\cF(x,y),z)$.
\end{enumerate}
We are interested in the graded setting, where $S$ is a $\mbbZ$-graded commutative ring, $x$ and $y$ are in degree $2$, and $\cF(x,y)$ is a homogenous degree $2$ power series.
Given a formal group law, there exists a unique inverse $\iota x \in S[[x]]$ such that $\cF(x, \iota x) = 0$. We will write $x +_\cF y$ for $\cF(x,y)$ and $x -_\cF y$ for $\cF(x, \iota y)$. 
\begin{example}[Additive formal group law]\label{ex:additive-formal-group-law}
    Let $S$ be a commutative ring, which we view as a graded commutative ring concentrated in degree $0$.
    The \emph{additive} formal group law is 
    \begin{equation}
        x +_\cF y \coloneqq x + y, \, \iota x = -x.
    \end{equation}
\end{example}
\begin{example}[Multiplicative formal group law]\label{ex:multiplicative-formal-group-law}
    Let $S= \mbbZ[\beta]$ with $\beta$ in degree $-2$. The \emph{multiplicative} formal group law is defined as
    \begin{equation}
        x +_{\cF} y \coloneqq x + y - \beta\, x\,y, \, \iota x =
        -\sum_{i \geq 0}\,\beta^i\, x^{i+1}.
    \end{equation}
\end{example}
\begin{example}[Lazard ring]\label{ex:universal-formal-group-law}
    There exists a graded commutative ring $L$, called the \emph{Lazard ring}, together with a formal group law $\cF$ on $L$ such that 
    for any graded commutative ring $R$, formal group laws on $R$ are in one-to-one correspondence with ring homomorphisms from $L$ to $R$. Categorically, $L$ is the initial object in the category of formal group laws. Lazard \cite{lazard} showed that $L$ is a free polynomial algebra $\mathbb{Z}[x_i]$ on infinitely many generators,
    where $i > 0$ and $x_i$ in degree $-2i$.
\end{example}
See \cite[Examples 3.3-3.6]{khovdef} for more examples.

We need the following lemma from \cite[Claim 3.8]{khovdef}:
\begin{lemma}\label{lem:unique-g}
   There exists a unique invertible element $g(x,y) \in S[[x,y]]$ such that 
   \begin{equation}
    x-y = (x-_\cF y) \, g(x,y).
   \end{equation}
\end{lemma}

Fix the pair $(S, \cF)$. Let $A$ be an abelian group, and $S[[A]]$ be the formal power series generated by  variables $x_\mu, \mu \in A$, each in degree two.  
We denote by $S[[A]]_\cF$ the quotient 
\begin{equation}
    S[[A]]/(x_0, x_{\mu + \nu} - (x_\mu +_F x_{\nu})).
\end{equation}
When $A \simeq \mbbZ^n$ is a finite rank lattice with basis $\mu_1, \cdots, \mu_n$, then we have 
\begin{equation}\label{eq:SA-and-power-series}
    S[[A]]_{\cF} \simeq S[[x_1, \cdots, x_n]]
\end{equation}
with $x_i$ corresponding to $x_{\mu_i}$.

The construction $S[[-]]_{\cF}$ defines a functor from the category of abelian groups to $S$-algebras. In particular, if there is a group $W$ that acts on the abelian group $A$, then $S[[A]]_{\cF}$ inherits a $W$ action by permuting the variables $x_{\mu}$.
Let $\Lambda = \mbbZ^n$, which we view as the root lattice of $\U(n)$. Using the standard basis, we can write $S[[\Lambda]]_{\cF}$ as $S[[x_1, \cdots, x_n]]$.
The Weyl group $W = S_n$ acts on $S[[\Lambda]]_{\cF}$ by permuting the variables $x_i$. For any simple transposition $s_i$, let 
\begin{equation}\label{eq:s_i-fixed-points}
    S[[\Lambda]]_{\cF}^{s_i} \simeq S[[x_1, \cdots, x_n]]^{s_i} =
S[[x_1, \cdots, x_{i-1}, x_{i}+ x_{i+1}, x_{i} x_{i+1}, x_{i+2},\cdots, x_{n}]
\end{equation}
 denote the subring of elements that are fixed by $s_i$. 
Here's an elementary but useful observation:
\begin{observation}\label{obs:S-free-rank-2}
    $S[[\Lambda]]_{\cF}$ is a free rank two $S[[\Lambda]]_{\cF}^{s_i}$-module with basis $1$ and $x_i$.
\end{observation}

Next we move on to Demazure (divided difference) operators. 
Note that simple roots of $\U(n)$ together with the vector $(1, 0, \cdots, 0)$ form a basis of $\Lambda$. 
It follows from  \cref{eq:SA-and-power-series} that the element $x_\alpha \in S[[\Lambda]]_{\cF} = S[[x_1, \cdots, x_n]]$ is regular for any simple root $\alpha$.\footnote{For more general root datum and $\alpha$ a simple long root, the element $x_{\alpha}$ may not be regular. See the discussion in \cite[\S 4]{CZZ}.}
By \cite[Lemma 2.3]{li2020equivariant}, for any simple transposition $s_i$, $\alpha_i$ the corresponding simple root, and any $r \in S[[\Lambda]]_{\cF}$, the power series $r - s_i(r)$ is divisible by $x_{\alpha_i} = x_{i} -_\cF x_{i+1}$ and $x_{-\alpha_i} = x_{i+1} -_\cF x_i$.
\begin{definition}\label{def:Demazure}
    Let $s_i$ be a simple transposition and $\alpha_i$ be the corresponding simple root. The Demazure operator is
    \begin{equation}\label{eq:Demazure}
        \partial_i \colon S[[\Lambda]]_{\cF} \to S[[\Lambda]]_{\cF}, \quad
        f \mapsto \frac{f - s_i(f)}{x_{\alpha_i}} =
        \frac{f - s_i(f)}{x_i -_\cF x_{i+1}} .
    \end{equation}
    We also have a variant 
\begin{equation}
    \partial'_i \colon  S[[\Lambda]]_{\cF} \to S[[\Lambda]]_{\cF}, \quad
    f \mapsto \frac{s_i(f) - f}{x_{-\alpha_i}} = 
    \frac{s_i(f) - f}{x_{i+1} -_\cF x_i}.
\end{equation}
\end{definition}

The following proposition follows from direct computation:
\begin{proposition}\label{prop:demazure-prop}
    The following holds:
    \begin{enumerate}
        \item $\partial_i r = 0 = \partial'_i r$ for any $r \in S[[\Lambda]]_{\cF}^{s_i}$.
        \item $\partial_i (r_1\, r_2) = (\partial_i r_1)\, r_2 + s_i(r_1) \,(\partial_i(r_2))$ and 
        $\partial'_i (r_1\, r_2) = (\partial'_i r_1)\, s_i(r_2) + r_1 \,(\partial'_i(r_2))$. 
        \item $\partial_i$ and $\partial'_i$ are $S[[\Lambda]]_{\cF}^{s_i}$-linear.
        \item $\partial_i x_i = g(x_i, x_{i+1}) = -\partial_i x_{i+1}$, where $g$ is defined in \cref{lem:unique-g}. Analogously, $\partial_i x_i = g(x_{i+1}, x_{i}) = -\partial_i x_{i+1}$.
    \end{enumerate}
\end{proposition}

\begin{remark}
    By \cref{obs:S-free-rank-2} and \cref{prop:demazure-prop}(3), $\partial_i$ and $\partial'_i$ are determined as  $S[[\Lambda]]_{\cF}^s$-linear maps by sending $1$ to $0$ and $x_i$ to the invertible elements $g(x_i, x_{i+1})$ and $g(x_{i+1}, x_i)$ respectively. 
\end{remark}

\subsection{Complex oriented ring spectra and cohomology of classifying spaces}\label{subsec:complex-orieted-ring-spectra}
Throughout the rest of the section we fix $E$ an $\EE_\infty$-ring spectrum 
with a complex $\EE_\infty$-orientation, that is, an $\EE_\infty$-ring map $f_E \colon \MU \to E$. We take this to be our notion of complex oriented ring spectra.\footnote{See \cite{hopkins2018strictly} for discussion of notions of complex orientation.}
In this subsection we review the relation between of complex oriented ring spectra and formal group laws, as well as the complex oriented cohomology of $B\U(n)$. 

Let $X$ be a space. We denote by $E(X) \coloneqq \Hom(\Sigma^{\infty}_+ X, E)$ the mapping spectrum whose homotopy groups are the $E$-cohomology $E^*(X)$ of $X$. If $X$ is a based space, we denote by $\tilde{E}(X) \coloneqq \Hom(\Sigma^{\infty}X, E)$ the mapping spectrum whose homotopy groups are the reduced $E$-cohomology $\tilde{E}^*(X)$ of $X$. Note that $E(X) \simeq \tilde{E}(X_+)$. We also denote the graded commutative ring $\pi_{-*}(E) = E^*(\pt)$ simply as $E^*$.
In particular, we use the cohomological grading. Lastly, let $X$ be a spectrum and $k \in \mbbZ$; we write $\Sigma^{-k}X$ as $X(k)$.\footnote{We use $(n)$ rather than $\Sigma$ because this suspension corresponds to the grading shift in Soergel bimodule. See \cref{rem:suspensions}.}

The orientation $f_E$ provides a nice theory of $E$-valued Chern classes for complex vector bundles. 
In particular, $E^*(B\U(1)) \simeq E^*[[x]]$, where the generator $x \in E^2(B\U(1))$ is the first Chern class $c_1(\cL)$ of the tautological line bundle $\cL$ on $B\U(1)$.
Furthermore, by the well-known result of Quillen \cite{MR0253350}, the tensor products of line bundles induce a formal group law $\cF_E$ on $E^*$:
consider the map $B\U(1) \times B\U(1) \to B\U(1)$ given by tensoring of line bundles. 
We can define $\cF_E$ as the image of $x$ under pullback on $E$-cohomology: 
\begin{equation}\label{eq:define-cF_E}
  E^*[[x]] \simeq  E^*(B\U(1)) \to E^*(B\U(1) \times B\U(1)) \simeq E^*[[x, y]], \quad x \mapsto \cF_E(x,y).
\end{equation}
We will often denote $\cF_E$ as $\cF$ when it is clear from context.
\begin{example}[Ordinary commutative ring]\label{ex:HS-additive}
    Let $S$ be a commutative ring. The associated Eilenberg-MacLane spectrum $HS$ is an $\EE_\infty$-ring spectrum with a canonical complex $\EE_\infty$-orientation $\MU \to \tau_{\leq 1} \MU \simeq H\mbbZ \to HS$. 
    The associated formal group law is the additive formal group law described in \cref{ex:additive-formal-group-law}. 
\end{example}
\begin{example}[K-theory]\label{ex:K-multiplicative}
    Let $ku$ be the connective complex $K$-theory spectrum. It has a canonical complex $\EE_\infty$-orientation $\MU \to ku$ and the associated formal group law is the multiplicative formal group law described in \cref{ex:multiplicative-formal-group-law}.
\end{example}
\begin{example}[$\MU$]\label{ex:MU-Lazard}
    The universal complex oriented ring spectrum is $\MU$ with the orientation $\id \colon \MU \to \MU$. 
    By a result of Milnor \cite{MR0119209}, $\MU^*$ is isomorphic to the Lazard ring $L$ described in \cref{ex:universal-formal-group-law}. Furthermore, in \cite{MR0253350} Quillen showed that the associated formal group law is indeed the universal formal group law on $L$.
\end{example}

Now we review the $E$-cohomology of classifying spaces of tori and $\U(n)$. Let $T$ be a torus and $\Lambda_T = \Hom(T, \U(1))$ be its character lattice. For $\alpha \in \Lambda_T$, let $\cL_\alpha$ be the associated complex line bundle on $BT$.
\begin{proposition}\label{prop:E-coho-on-torus}
    There exists a functorial equivalence
    \begin{equation}
        E^*(BT) \simeq E^*[[\Lambda]]_{\cF_E}
    \end{equation}
    where $c_1(\cL_\alpha)$ is taken to $x_\alpha$ for $\alpha \in \Lambda$.
\end{proposition}
Now for the cohomology of $B\U(n)$ and more generally parabolic subgroups of $\U(n)$:
\begin{proposition}\label{prop:E-coho-on-parabolics}
    Fix $n \in \mathbb{N}$. Let $W' \subset W=S_n$ be a parabolic subgroup and $G_{W'} \subset G = \U(n)$ be the corresponding standard parabolic.
    The inclusion $T = \U(1)^n \to G_{W'}$ induces a map $E^*(BG_{W'}) \to E^*(BT)$. This map is injective and induces an isomorphism:
    \begin{equation}
        E^*(BG_{W'}) \simeq E^*(BT)^{W'}  \simeq E^*[[\Lambda]]_{\cF_E}^{W'}.
    \end{equation}
\end{proposition}
\begin{example}\label{ex:E-coho-on-BUn}
    It follows that $E^*(B\U(n)) \simeq E^*[[c_1, \cdots, c_n]] \subset E^*[[x_1, \cdots, x_n]] \simeq E^*(BT)$, where $c_i \in E^{2i}(B\U(n))$ is the $i$-th symmetric polynomial in the variables $x_1 \cdots, x_n$. Furthermore, $c_i$ is indeed the $i$-th Chern class of the tautological vector bundle over $B\U(n)$.
\end{example}
\begin{example}\label{ex:E-of-BGi}
    Fix $1 \leq i \leq n-1$. \cref{prop:E-coho-on-parabolics} implies that 
    $$E^*(BG_i) \simeq E^*[[\Gamma]]^{s_i}_{\cF} = E^*[[x_1, \cdots, x_{i-1}, x_{i}+ x_{i+1}, x_{i} x_{i+1}, x_{i+2},\cdots, x_{n}].$$ 
    In particular, by \cref{obs:S-free-rank-2}, $E^*(BT)$ is a rank two free module over $E^*(BG_i)$ with basis $1$ and $x_i$.
\end{example}
We also get the Kunneth isomorphism:
\begin{proposition}\label{prop:kunneth-theorem}
    The Kunneth map 
    \begin{equation}
 E^*(B\U(n_1)) \otimes_{E^*} \cdots \otimes_{E^*} E^*(B\U(n_k))
\to
E^*(B\U(n_1) \times \cdots B\U(n_k))
    \end{equation}
    is an isomorphism.
\end{proposition}

Lastly we review of orientations and Thom isomorphisms for complex vector bundles. We refer the reader to \cite{MR3252967, MR3286898} for more details.
As $\MU$ has the universal orientation for complex vector bundles, the map $f_E \colon \MU \to E$ induces an $E$-orientation for complex vector bundles. 
Orientations give rise to Thom isomorphisms:
\begin{proposition}\label{prop:E-thom-class}
   Let $X$ be a space and $V$ a rank $n$ complex vector bundle on $X$. We have  an equivalence of $E(X)$-module spectra 
   \begin{equation}
    \label{eq:thom-iso}
    \tilde{E}(X^V) \simeq  E(X)(2n).
   \end{equation}
   On cohomology, the unit $1 \in E^0(X)$ corresponds to the Thom class $\th(V) \in E^{2n}(X^V)$. Furthermore, the zero section inclusion induces a map of $E(X)$-module spectra
   \begin{equation}
    E(X)(2n) \simeq \tilde{E}(X^V) \to E(X).
   \end{equation}
  On cohomology this takes $1$ to the top Chern class $c_n(V) \in E^{2n}(X)$. 
\end{proposition}
Moving forward, we will always use the Thom isomorphism \eqref{eq:thom-iso} to identify $\tilde{E}(X^V)$ with $E(X)(2n)$.

\subsection{$E$-valued Bott-Samelson bimodules}\label{subsec:E-valued-BS}
In this subsection we define $E$-valued Bott-Samelson bimodules.
Given $H$ a group and $X$ a $H$-space, $E(X/H)$ is naturally an $\EE_\infty$-algebra over $E(BH)$ via pulling back along the map $X/H \to \pt/H = BH$. Note that here $X/H$ is the homotopy (stacky) quotient. 
In our case, if $X$ is a $H-H$-space, then $E(H \backslash X/H)$ is an $E(BH)-E(BH)$-bimodule spectrum.\footnote{Please note that by $E(BH)-E(BH)$-bimodule, we are referring to an $E(BH)-E(BH)$-bimodule $E$-module spectrum, meaning the left and right $E$ action is identified. We will use this simplified terminology in this section.}

Throughout the rest of the section we fix $n \in \mbbN$,  $G = \U(n)$, and $T = \U(1)^n$. 
\begin{definition}\label{def:E-Bott-Samelson}
    Fix  $\underbf{i} = (i_1, \cdots, i_m)$. The $E$-valued Bott-Samelson bimodule associated to $\underbf{i}$ is 
   \begin{equation}
    \EBS_{\underbf{i}} \coloneqq E(T\backslash \BS_{\underbf{i}}/T),
   \end{equation}
   which we view as an $E(BT)-E(BT)$-bimodule. We denote $\EBS_{\underline{i}}$ by $\EBS_{i}$.
\end{definition}
We will also denote the cohomology groups $E^*(T\backslash \BS_{\underbf{i}}/T)$ as $\EBS^*_{\underbf{i}}$, and $E^*(T\backslash \BS_i /T)$ as $\EBS^*_i$.
\begin{observation}\label{obs:tensor-over-EBT}
    Let $- \otimes_{E(BT)} -$ denote tensoring over $E(BT)$; then we have an equivalence 
    \begin{equation}
        \EBS_{\underbf{i}} \simeq \EBS_{i_1} \otimes_{E(BT)} \EBS_{i_2} \cdots  \otimes_{E(BT)} \EBS_{i_m}.
    \end{equation}
\end{observation}

Let us observe the following: suppose we have a group $G$ and a subgroup $H \subset G$; then the double qoutient $H\backslash G/H$ is equivalent to the pullback $BH \times_{BG} BH$. 
In our case, we see that 
$T\backslash G_i/T \simeq BT \times_{BG_i} BT$. More generally, we have 
\begin{equation}\label{eq:iterative-def-of-TT-Soergel}
    T\backslash \BS_{\underbf{i}}/T \simeq BT \times_{BG_{i_1}} BT \cdots \times_{BG_{i_m}} BT.
\end{equation}
Now we can compute $\EBS^*_{\underbf{i}}$:
\begin{proposition}\label{prop:E-coho-of-BS}
    The canonical map on cohomology induces by \eqref{eq:iterative-def-of-TT-Soergel}
    \begin{equation}\label{eq:EM-collapse}
        E^*(BT)\otimes_{E^*(BG_{i_1})} E^*(BT) \cdots \otimes_{E^*(BG_{i_m})} E^*(BT) \to E^*(T\backslash \BS_{\underbf{i}} /T) = \EBS^*_{\underbf{i}}
    \end{equation}
    is an isomorphism. The left and right $E^*(BT)$ action is on the leftmost and rightmost factors respectively.
\end{proposition}

\begin{proof}
    First consider the case that $\underbf{i} = i$ for some $1 \leq i \leq n-1$, in which case $E^*(BT)$ is a free $E^*(BG_{i}) \simeq E^*(BT)^{s_i}$-module by \cref{ex:E-of-BGi}. It follows that the Eilenberg-Moore spectral sequence degenerates and the canonical map 
    $E^*(BT)\otimes_{E^*(BG_{i})} E^*(BT) \to \EBS^*_i$ is an isomorphism.
    The general case follows inductively from the same argument.
\end{proof}
By \cref{prop:E-coho-on-parabolics}, we have the following:
\begin{corollary}
    We have an isomorphism of graded rings:
   \begin{equation}
    \EBS^*_{\underbf{i}} \simeq  E^*[[\Lambda]]_{\cF} \otimes_{E^*[[\Lambda]]_{\cF}^{s_1}} E^*[[\Lambda]]_{\cF} \cdots \otimes_{E^*[[\Lambda]]_{\cF}^{s_m}} E^*[[\Lambda]]_{\cF}.
   \end{equation}
\end{corollary}
\begin{example}
    Fix $n = 2$ and $s_1 = (1,2)$. We have $E^*(BT) \simeq E^*[[x_1, x_2]]$ and 
    \begin{equation}
        \begin{aligned}
            \EBS^*_1 &= E^*(BT)\otimes_{E^*(BT)^{s_1}}E^*(BT) 
            \\ &= E^*[[x_1 \otimes 1, x_2 \otimes 1, 1 \otimes x_1, 1 \otimes x_2]]/((x_1 + x_2) \otimes 1 = 1 \otimes (x_1 + x_2), (x_1 x_2) \otimes 1= 1 \otimes (x_1 x_2)).     
        \end{aligned}
    \end{equation}
\end{example}

We end this subsection by translating \cref{prop:two-fiber-sequences} 
and \cref{cor:four-fiber-sequences} to this setting. 
Note that the double quotient $BTs_i \coloneqq T\backslash Ts_i/T$ is equivalent to $BT$ as a space. However, the associated map to $T\backslash \pt /T = BT \times BT$ is $(\id, \sigma_i)$.\footnote{Recall that $\sigma_i \colon BT \to BT$ is the automorphism given by conjugation by $s_i$.} It follows that the $E^*(BT)-E^*(BT)$ action on $E^*(BTs_i)$ is 
\begin{equation}
    (f_0, f_1) \cdot r = f_0\, r \,s_i(f_1),
\end{equation}
where $s_i(-)$ exchanges variables $x_i$ and $x_{i+1}$. 

Furthermore, we need to consider equivariant cohomology of Thom spaces:
suppose $H$ is a group and $X$ is a $H$-space.
In addition, suppose $V$ is a $H$-equivariant vector bundle on $X$, so it descends to a vector bundle $V/H$ over $X/H$. Furthermore, $\tilde{E}_H(X^V) \simeq \tilde{E}((X/H)^{V/H})$ is naturally an $E(BH)$-module.
In our case, we see that $T^{\cL_\alpha}$ and $(Ts_i)^{\cL_{-\alpha}}$ correspond to $BT^{\alpha}$ and $(BTs_i)^{-\alpha}$. Using the complex orientation (\cref{prop:E-thom-class}), we can identify
$\tilde{E}(BT^{\alpha})$ and $\tilde{E}((BTs_i)^{-\alpha})$ as $E(BT)(2)$ and $E(BTs_i)(2)$ respectively.

For the rest of the subsection we fix $1 \leq i \leq n-1$.
We get an $E$-valued version of \cref{prop:two-fiber-sequences}:
\begin{proposition}\label{prop:E-two-fiber-seq}
    We have maps of $E(BT)-E(BT)$-bimodule spectra 
    \begin{equation}\label{eq:E-fancy-diagram}
        \begin{tikzcd}
            E(BT)(2)\ar[rd, "\Delta_i"] & & E(BTs_i)(2)\ar[ld, "\Deltas_i"']\\
            & \EBS_i \ar[rd, "ms_i"] \ar[ld, "m_i"'] & \\
            E(BT) & & E(BTs_i).
        \end{tikzcd}
    \end{equation}
    Furthermore, the two diagonals are fiber sequences, while the two other composites $E(BT)(2) \to E(BT)$ and $E(BTs_i)(2) \to E(BTs_i)$ are multiplications by $c_1(\alpha_i) = x_i -_{\FE }x_{i+1}$ and $c_1(-\alpha_i) = x_{i+1} -_{\FE} x_{i}$ respectively. 
\end{proposition}
Next we determined what these maps in \eqref{eq:E-fancy-diagram} are on cohomology:
\begin{proposition}\label{prop:E-coho-two-fiber-seq}
    The induced maps on cohomology from  \eqref{eq:E-fancy-diagram} are given by 
    \begin{equation}\label{eq:E-coho-fancy-diagram}
        \begin{tikzcd}
            r \ar[rd, mapsto] &  &r \ar[ld, mapsto] \\ 
            & r\, (x_i \otimes 1 -_\FE 1 \otimes x_{i+1}) \quad \quad  r \,(x_{i+1} \otimes 1 -_\FE 1 \otimes x_{i+1}) &  \\ 
            & r_1 \otimes r_2 \ar[rd, mapsto] \ar[ld, mapsto] \\ 
            r_1 \,r_2 & & r_1 \,s_1(r_2).
        \end{tikzcd}
    \end{equation}
\end{proposition}
\begin{proof}
    The forms of the bottom two maps are determined by being ring homomorphisms and $E^*(BT)-E^*(BT)$ bimodule maps.
    Furthermore, the combined map $EB^*_i \to E^*(BT) \oplus E^*(BTs_i)$ is injective by \cite[Theorem 1.1]{li2020equivariant}.\footnote{In fact the image can be characterized by the Goresky-Kottwitz-MacPherson description, see \cite{li2020equivariant}.}
    Therefore to determine the top two maps it suffices to check the composites to $E^*(BT)$ and $E^*(BTs_i)$. 
    Let us first consider $\Delta_i \colon E^*(BT)(2) \to \EBS_{i}^*$. By \cref{prop:E-two-fiber-seq}, the composite $ms_i \circ \Delta_i$ is the $0$, while $\Delta_i \circ m_i$ is given by multiplication by $x_i -_\FE x_{i+1}$. It is straightforward to see that $r \mapsto r \,(x_{i} \otimes 1 -_\FE 1 \otimes x_{i+1})$ satisfies both conditions. The same argument holds for $\Deltas_i \colon E^*(BTs_i)(2) \to \EBS^*$.
\end{proof}
\begin{remark}
    Note that $$x_i \otimes 1 -_\FE 1 \otimes x_{i+1} = 1 \otimes x_{i} -_\FE x_{i+1} \otimes 1$$ and $$x_{i+1} \otimes 1 -_\FE 1 \otimes x_{i+1} = 1 \otimes x_i -_\FE x_i \otimes 1.$$ This can be verified by checking their images in $E^*(BT)$ and $E^*(BTs_i)$.   
\end{remark}

\begin{remark}
    Since $\Delta_i \colon E^*(BT)(2) \to \EBS_i^*$ is an $E^*(BT)-E^*(BT)$-equivariant map, we see that for any $r \in E^*(BT) = E^*[[x_1, x_2]]$, we have 
    \begin{equation}\label{eq:hochschild-1}
        \begin{aligned}
            r \otimes 1 \times (x_i \otimes 1 -_\FE 1 \otimes x_{i+1}) 
            &= r\, (x_i \otimes 1 -_\FE 1 \otimes x_{i+1})
            \\ & = (x_i \otimes 1 -_\FE 1 \otimes x_{i+1})\, r   
            \\ &= (x_i \otimes 1 -_\FE 1 \otimes x_{i+1}) \times 1 \otimes r.        
        \end{aligned}
    \end{equation}
    and 
    \begin{equation}\label{eq:hochschild-2}
        \begin{aligned}
            r \otimes 1 \times (x_{i+1} \otimes 1 -_\FE 1 \otimes x_{i+1}) &= 
            r\, (x_{i+1} \otimes 1 -_\FE 1 \otimes x_{i+1}) 
            \\ &= (x_{i+1} \otimes 1 -_\FE 1 \otimes x_{i+1})\, s_i(r)
            \\ & = (x_{i+1} \otimes 1 -_\FE 1 \otimes x_{i+1}) \times 1 \otimes s_i(r).
        \end{aligned}
    \end{equation}
    Here $\times$ denotes the multiplication of $\EBS^*_i$.
\end{remark}

The isomorphisms in \eqref{eq:non-obvious-iso} induces isomorphisms:
\begin{equation}\label{eq:E-non-obvious-iso}
    \begin{tikzcd}
     \EBS_i \otimes_{E(BT)} E(BTs_i) 
    \ar[r, "\sim"] &  \EBS_i & \ar[l, "\sim"'] 
    E(BTs_i) \otimes_{E(BT)} \EBS_i
    \\ 
        r_1 \otimes r_2 \otimes t  = r_1 \otimes r_2 \,t \otimes 1 \ar[r, mapsto] & r_1 \otimes s_i\,(r_2\, t) & \\
        & s_i(t)\, r_1 \otimes r_2 & \ar[l, mapsto] t \otimes r_1 \otimes r_2 = 1 \otimes s_i(t)\, r_1 \otimes r_2.
    \end{tikzcd}
\end{equation}

Finally we have the $E$-valued version of \cref{cor:four-fiber-sequences}:
\begin{corollary}\label{cor:E-four-fiber-sequences}
    We have a diagram
\begin{equation}\label{eq:E-fancy-diagram-times-Gs-1}
    \begin{tikzcd}
        \EBS_i(2) \ar[rrd, "\Delta_i \otimes_{E(BT)} \EBS_i"']& && & \EBS_i(2)\ar[lld, "\Deltas_i \otimes_{E(BT)} \EBS_i"] \\ 
        && \EBS_{\underline{ii}} \ar[rrd, "ms_i \otimes_{E(BT)} \EBS_i"]
        \ar[lld, "m_i\\\ \otimes_{E(BT)} \EBS_i"'] && \\
        \EBS_i &&& & \EBS_i 
    \end{tikzcd}
\end{equation}
where the diagonal maps are fiber sequences. Note that we twisted $ms_i \otimes_{E(BT)} \EBS_i$ and $\Deltas_i \otimes_{E(BT)} \EBS_i$ by the isomorphisms in \eqref{eq:E-non-obvious-iso}.
Furthermore, the induced maps on cohomology are given by 
\begin{equation}\label{eq:E-coho-fancy-diagram-Gs-1}
    \begin{tikzcd}[scale cd=.90]
        r_1 \otimes r_2 \ar[rd, mapsto] &  &r_1 \otimes r_2 \ar[ld, mapsto] \\ 
        & r_1 \,(x_i \otimes 1 -_\FE 1 \otimes x_{i+1}) \otimes r_2 \quad \quad  r_1\, (x_{i+1} \otimes 1 -_\FE 1 \otimes x_{i+1}) \otimes r_2 &  \\ 
        & r_1 \otimes r_2 \otimes r_3 \ar[rd, mapsto] \ar[ld, mapsto] \\ 
        r_1 \,r_2 \otimes r_3 & & r_1 \,s_i(r_2) \otimes r_3.
    \end{tikzcd}
\end{equation}

Similarly we have a diagram
\begin{equation}\label{eq:E-fancy-diagram-times-Gs-2}
    \begin{tikzcd}
        \EBS_i(2) \ar[rrd, " \EBS_i  \otimes_{E(BT)} \Delta_i "']& && & \EBS_i(2)\ar[lld, "\EBS_i  \otimes_{E(BT)}   \Deltas_i "] \\ 
        && \EBS_{\underline{ii}} \ar[rrd, "\EBS_i  \otimes_{E(BT)} ms_i "]
        \ar[lld, "\EBS_i  \otimes_{E(BT)} m_i "'] && \\
        \EBS_i &&& & \EBS_i 
    \end{tikzcd}
\end{equation}
where the diagonal maps are fiber sequences. 
Once again we twisted $\EBS_i \otimes_{E(BT)} ms_i $ and $\EBS_i \otimes_{E(BT)} \Deltas_i $ by the isomorphisms in \eqref{eq:E-non-obvious-iso}.
Furthermore, the induced maps on cohomology are given by
\begin{equation}\label{eq:E-coho-fancy-diagram-Gs-2}
    \begin{tikzcd}[scale cd=.90]
        r_1 \otimes r_2 \ar[rd, mapsto] &  &r_1 \otimes r_2 \ar[ld, mapsto] \\ 
        & r_1  \otimes (x_i \otimes 1 -_\FE 1 \otimes x_{i+1})\, r_2 \quad \quad  r_1  \otimes (x_{i+1} \otimes 1 -_\FE 1 \otimes x_{i+1})\, r_2  &  \\ 
        & r_1 \otimes r_2 \otimes r_3 \ar[rd, mapsto] \ar[ld, mapsto] \\ 
        r_1 \otimes r_2 \,r_3 & & r_1 \otimes s_i(r_2)\,  r_3.
    \end{tikzcd}
\end{equation}
\end{corollary}
Here we used \cref{prop:E-coho-two-fiber-seq} and \eqref{eq:E-non-obvious-iso} to compute the maps on cohomology.



\subsection{Splittings of $E$-valued Bott-Samelson bimodules}\label{subsec:splittings-of-EBS}
Fix $i, j$ neighboring simple transpositions.
In this subsection we construct splittings of $\EBS_{\underline{ii}}$ and $\EBS_{\underline{iji}}$. We start with $\EBS_{\underline{ii}}$. Recall that we have a $T-T$-equivariant multiplication map $\mu_i \colon G_i \times_T G_i \to G_i$. This induces a map of $E(BT)-E(BT)$-bimodules:
\begin{equation}\label{eq:E-mu-mult}
    \mu_i \colon \EBS_i \to \EBS_{\underline{ii}}, \quad r_1 \otimes r_2 \mapsto r_1 \otimes 1 \otimes r_2.
\end{equation}

\begin{proposition}\label{prop:E-BS_s-splitting-1}
    We have a splitting
    \begin{equation}\label{eq:E-BS_s-splitting-1}
        \begin{tikzcd}[column sep=1.5cm]
    \EBS_i
    \arrow[yshift=0.9ex]{rr}{\mu_i}
\arrow[leftarrow, yshift=-0.9ex, "m_i \otimes_{E(BT)}\EBS_i"']{rr}[yshift=-0.2ex]{}
&&
\EBS_{\underline{ii}}
\arrow[yshift=0.9ex]{rr}{\nabla_i^L}
\arrow[leftarrow, yshift=-0.9ex, "\Deltas_i \otimes_{E(BT)}\EBS_i"']{rr}[yshift=-0.2ex]{}
&&
\EBS_i(2)
\\
r_1 \otimes r_2 \ar[rr] && r_1 \otimes 1 \otimes r_2 && 
\\
r_1  \, r_2 \otimes r_3 && \ar[ll] r_1 \otimes r_2 \otimes r_3 \ar[rr] && 
r_1 \, \partial'_i r_2  \otimes r_3 \\ 
&& 
r_1 \, (x_{i+1} \otimes 1 -_\FE 1 \otimes x_{i+1}) \otimes  r_2 
&& \ar[ll] r_1 \otimes r_2
        \end{tikzcd}
    \end{equation}
\end{proposition}
Recall that $\partial'_i r \coloneqq \frac{s_i(r)- r}{x_{i+1}-_{\FE} x_{i}}$
is defined in \cref{def:Demazure}.

\begin{proof}
    The splitting follows from  \cref{obs:mult-splitting}.
    We only need to verify the maps on cohomology. As three  of the four maps are described in \eqref{eq:E-coho-fancy-diagram-Gs-1} and \eqref{eq:E-mu-mult}, it is enough to check that 
    the projection $r_1 \otimes r_2 \otimes r_3 \mapsto r_1 \,\partial_i' r_s \otimes r_3$ is the correct map. That is, we need to check that 
    \begin{equation}
        r_1 \otimes r_2 \otimes r_3 = r_1 \,r_2 \otimes 1 \otimes r_3 + r_1 \, \partial_i' r_2 \,(x_{i+1} \otimes 1 -_\FE 1 \otimes x_{i+1}) \otimes r_3.
    \end{equation}

    It is clear that we can take $r_1 = r_3 = 1$, in which case it reduces to showing that 
    \begin{equation}\label{eq:need-to-check-this}
       1 \otimes r_2 - r_2 \otimes 1 = \partial_i' r_2 \,(x_{i+1} \otimes 1 -_\FE 1 \otimes x_{i+1})
    \end{equation}
    in $\EBS^*_i = E^*(BT) \otimes_{E^*(BT)^{s_i}} E^*(BT)$.
    Furthermore, as both sides are $E^*(BT)^{s_i}$-linear (\cref{prop:demazure-prop}), and $E^*(BT)$ is generated by $1$ and $x_i$ as a left $E^*(BT)^{s_i}$-module (\cref{ex:E-of-BGi}), it suffices to consider the cases $r_2 = 1$ and $r_2 = x_i$. For $r_2 = 1$, \eqref{eq:need-to-check-this} is clear as $\partial_i'\, 1 = 0$.
    For $r_2 = x_i$, note that 
    $1 \otimes x_i - x_i \otimes 1 = x_{i+1} \otimes 1 - 1 \otimes x_{i+1}$. 
    By definition of $g(-, -)$ in \cref{lem:unique-g}, we have 
    \begin{equation}
        x_{i+1} \otimes 1 - 1 \otimes x_{i+1} = g(x_{i+1} \otimes 1, 1 \otimes x_{i+1}) \times (x_{i+1} \otimes 1 -_\FE 1 \otimes x_{i+1}).
    \end{equation}
    By \eqref{eq:hochschild-2}, we can exchange a right multiplication by $1 \otimes x_{i+1}$ on $(x_{i+1} \otimes 1 -_\FE 1 \otimes x_{i+1})$ with a left multiplication by $x_{i} \otimes 1$.
    Therefore 
    \begin{equation}
        \begin{aligned}
            g(x_{i+1} \otimes 1, 1 \otimes x_{i+1}) \times (x_{i+1} \otimes 1 -_\FE 1 \otimes x_{i+1}) &= g(x_{i+1} \otimes 1, x_{i} \otimes 1)\times (x_{i+1} \otimes 1 -_\FE 1 \otimes x_{i+1}) \\
            &= g(x_{i+1}, x_i)\, (x_{i+1} \otimes 1 -_\FE 1 \otimes x_{i+1}) \\ 
            &= \partial'_i x_i\, (x_{i+1} \otimes 1 -_\FE 1 \otimes x_{i+1}),
        \end{aligned}
    \end{equation}
    where the last equality is given by \cref{prop:demazure-prop}(4).
\end{proof}

Analogously, we can use $\mu_i$ to split $\EBS_i \otimes_{E(BT)} m_i$:
\begin{proposition}\label{prop:E-BS_s-splitting-2}
    We have a splitting
    \begin{equation}\label{eq:E-BS_s-splitting-2}
        \begin{tikzcd}[column sep=1.5cm]
    \EBS_i
    \arrow[yshift=0.9ex]{rr}{\mu_i}
\arrow[leftarrow, yshift=-0.9ex, "\EBS_i \otimes_{E(BT)}m_i "']{rr}[yshift=-0.2ex]{}
&&
\EBS_{\underline{ii}}
\arrow[yshift=0.9ex]{rr}{\nabla^R_i}
\arrow[leftarrow, yshift=-0.9ex, "\EBS_i \otimes_{E(BT)} \Deltas_i "']{rr}[yshift=-0.2ex]{}
&&
\EBS_i(2)
\\
r_1 \otimes r_2 \ar[rr] && r_1 \otimes 1 \otimes r_2 && 
\\
r_1  \otimes r_2  \, r_3 && \ar[ll] r_1 \otimes r_2 \otimes r_3 \ar[rr] && 
r_1   \otimes (-\partial'_i r_2)\, r_3 \\ 
&& 
r_1  \otimes (x_{i+1} \otimes 1 -_\FE 1 \otimes x_{i+1})\, r_2 
&& \ar[ll] r_1 \otimes r_2
        \end{tikzcd}
    \end{equation} 
\end{proposition}
The proof is analogous to the proof of \cref{prop:E-BS_s-splitting-1}.
\begin{remark}
    By \cref{obs:mult-splitting}, $\mu_i$ also splits $\EBS_i \otimes_{E(BT)} ms_i$ and $\EBS_i \otimes_{E(BT)} ms_i$. The induced maps $\EBS_{\underline{ii}} \to \EBS_i(2)$ use the Demazure operator $\partial_i$ instead of $\partial'_i$.
\end{remark}
We also need the following lemma:
\begin{lemma}\label{lem:need-this-for-condition-2}
   The composite 
   \begin{equation}
   \EBS_i(2) \xrightarrow{\EBS_i \otimes_{E(BT)} \Deltas_i} \EBS_{\underline{ii}} \xrightarrow{\nabla_i^L} \EBS_i(2)
   \end{equation}
   is an equivalence.
\end{lemma}
\begin{proof}
    One can do this by explicit computation. However, we will give an abstract argument: since $\nabla_i^L, \nabla_i^R \colon \EBS_{\underline{ii}} \to \EBS_i(2)$ are both cofibers of $\EBS_i \xrightarrow{\mu} \EBS_{\underline{ii}}$, there is an automorphism $\phi \colon \EBS_i(2) \xrightarrow{\simeq} \EBS_i(2)$ such that $\nabla_i^L = \nabla_i^R \circ \phi$. Finally, we have
    \begin{equation}
       (\EBS_i \otimes_{E(BT)} \Deltas_i) \circ \nabla_i^L = 
        (\EBS_i \otimes_{E(BT)} \Deltas_i) \circ \nabla_i^R \circ \phi = \phi, 
    \end{equation}
    which is invertible.
\end{proof}

Now we move on to the the splitting of $\EBS_{\underline{iji}}$, where $i$ and $j$  are neighboring simple transpositions. 
Recall that $G_{i,j}$ is the standard parabolic subgroup corresponding to the $S_3 \subset W$ generated by $i$ and $j$.
\begin{notation}\label{nota:EBS_iji}
    Let $\EBS_{i,j}$ denote the $E(BT)-E(BT)$-bimodule $E(T\backslash G_{i,j}/T)$. We denote its cohomology by $\EBS^*_{i,j}$.
\end{notation}
By the same argument as \cref{prop:E-coho-of-BS}, we get the following:
\begin{lemma}\label{lem:E-coho-of-sts}
    $\EBS^*_{i,j} \simeq E^*(BT) \otimes_{E^*(BT)^{S_3}} E^*(BT)$.
\end{lemma}
The multiplication map $\mu_{iji} \colon \BS_{\underline{iji}} \to G_{i,j}$ induces a map $$\mu_{iji} \colon \EBS_{i,j} \to \EBS_{\underline{iji}}, \quad r_1 \otimes r_2 \mapsto r_1 \otimes 1 \otimes 1 \otimes r_2.$$

By \cref{prop:sts-pushout}, we get the following:
\begin{proposition}\label{prop:E-sts-pushout}
    We have a pullback (equivalently pushout) square of $E(BT)-E(BT)$-bimodule spectra:
    \begin{equation}\label{eq:E-sts-pushout}
        \begin{tikzcd}
            \EBS_{i,j} \ar[r] \ar[d, "\mu_{iji}"]& \EBS_i \ar[d, "\mu_i"]\\ 
            \EBS_{\underline{iji}} \ar[r] & \EBS_{\underline{ii}} \, .
        \end{tikzcd}
    \end{equation}
\end{proposition}
By \cref{prop:E-BS_s-splitting-1}, the cofiber of $\mu_i \colon \EBS_i \to \EBS_{\underline{ii}}$ is $\EBS_i(2)$. Since parallel maps in a pushout diagram have the same cofibers, we get the following:
\begin{corollary}\label{cor:E-fiber-sequence-sts}
    We have a fiber sequence 
    \begin{equation}\label{eq:E-fiber-sequence-sts}
       \EBS_{i,j} \xrightarrow{\mu_{iji}} \EBS_{\underline{iji}}  \to \EBS_i(2), 
    \end{equation}
    where the second map $\EBS_{\underline{iji}} \to \EBS_i(2)$ is the composite 
    \begin{equation}
       \EBS_{\underline{iji}} \xrightarrow{\EBS_i \otimes_{E(BT)} m_j \otimes_{E(BT)}\EBS_i}  \EBS_{\underline{ii}} \xrightarrow{\nabla_i^L}  \EBS_i(2).
    \end{equation}
    Similarly, there is another fiber sequence of the form \eqref{eq:E-fiber-sequence-sts} where the second map is the composite 
         \begin{equation}\label{eq:E-fiber-sequence-sts-2}
        \EBS_{\underline{iji}} \xrightarrow{\EBS_i \otimes_{E(BT)} m_j \otimes_{E(BT)}\EBS_i}  \EBS_{\underline{ii}} \xrightarrow{\nabla_i^R}  \EBS_i(2). 
    \end{equation}
\end{corollary}

Following standard Soergel bimodule theory, we expect a splitting 
$\EBS_{\underline{iji}} = \EBS_{i,j} \oplus \EBS_i(2)$:
\begin{proposition}\label{prop:E-sts-splitting-1}
    There exists a map $\EBS_i(2) \to \EBS_{\underline{iji}}$ that splits \eqref{eq:E-fiber-sequence-sts}.
\end{proposition}
\begin{proof}
    Let us assume that $j = i+1$. The other case $j = i-1$ is analogous.
    First we construct a map $f \colon \EBS_i(2) \to \EBS_{\underline{iji}}$ such that the composite $\EBS_i(2) \xrightarrow{f} \EBS_{\underline{iji}} \to \EBS_i(2)$ is an isomorphism, which we will check on cohomology.
    On cohomology, the map $\EBS_{\underline{iji}} \to \EBS_i(2)$ takes 
    $r_1 \otimes r_2 \otimes r_3 \otimes r_4 \mapsto r_1 \, \partial_i'(r_2\, r_3) \otimes r_4$. Let $f$ be the composite 
    \begin{equation}
        \begin{tikzcd}
            \EBS_i(2) \ar[r, "\mu_i"] &\EBS_{\underline{ii}}(2) \ar[r, "\EBS_i \otimes_{E(BT)} m_j \otimes_{E(BT)}\EBS_i"] &\EBS_{\underline{iji}} \\ 
            r_1 \otimes r_2 \ar[r, mapsto] & r_1 \otimes 1 \otimes r_2 \ar[r, mapsto] & r_1 \otimes (x_{i+1}\otimes 1 -_\FE 1 \otimes x_{i+2}) \otimes r_2.
        \end{tikzcd}
    \end{equation}
    The total composite $\EBS_i(2) \xrightarrow{f} \EBS_{\underline{iji}} \to \EBS_i(2)$ takes 
    \begin{equation}
        r_1 \otimes r_2 \mapsto r_1 \,\partial_i' (x_{i+1} -_\FE x_{i+2}) \otimes r_2.
    \end{equation}
    Therefore it suffices to check that $\partial_i' (x_{i+1} -_\FE x_{i+2})$ is an invertible element in $E^*(BT)$, which follows from the following calculation:
    \begin{equation}
        \begin{aligned}
            \partial_i' (x_{i+1} -_\FE x_{i+2}) &= 
            \frac{(x_{i} -_\FE x_{i+2}) - (x_{i+1} -_\FE x_{i+2})}{x_{i+1}-_\FE x_{i}}\\ 
            &=  \frac{(x_{i} -_\FE x_{i+2}) -_\FE (x_{i+1} -_\FE x_{i+2})}{x_{i+1}-_\FE x_{i}} \,g(x_{i} -_\FE x_{i+2},x_{i+1} -_\FE x_{i+2}) \\ 
            &= \frac{x_{i} -_\FE x_{i+1}}{x_{i+1} -_\FE x_i} \,g(x_{i} -_\FE x_{i+2},x_{i+1} -_\FE x_{i+2}) \\ 
            &= \frac{x_i - x_{i+1}}{x_{i+1}-x_i}\, g^{-1}(x_i, x_{i+1})
           \,  g(x_{i+1}, x_i)\, g(x_{i} -_\FE x_{i+2},x_{i+1} -_\FE x_{i+2}) \\ 
           &= -g^{-1}(x_i, x_{i+1})
           \,  g(x_{i+1}, x_i)\, g(x_{i} -_\FE x_{i+2},x_{i+1} -_\FE x_{i+2}).
        \end{aligned}
    \end{equation}
    Note that $g$ is invertible by \cref{lem:unique-g}.

    Since the composite $\EBS_i(2) \xrightarrow{f} \EBS_{\underline{iji}} \to \EBS_i(2)$ is an isomorphism, it has an inverse $\phi$. It follows that $f \circ \phi$ is a section of $\EBS_{\underline{iji}} \to \EBS_i(2)$.
\end{proof}
By an analogous argument, we have the following:
\begin{corollary}\label{cor:E-sts-splitting-2}
    There exists a map $\EBS_i(2) \to \EBS_{\underline{iji}}$ that splits  \eqref{eq:E-fiber-sequence-sts-2}.
\end{corollary}

We conclude this section with the following simple observation:
\begin{observation}\label{obs:E-sts-commuting-square}
    The commuting square of $T-T$ spaces
    \begin{equation}
        \begin{tikzcd}
            G_i \times_T G_j \ar[rrr, "G_i \times_T G_j \times_T m_i"] \ar[d, "m_j \times G_i \times_T G_j"] &&& G_i \times_T G_j \times G_i \ar[d, "\mu_{iji}"] \\ 
            G_j \times_T G_i \times_T G_j \ar[rrr, "\mu_{jij}"] &&& G_{i,j}
        \end{tikzcd}
    \end{equation}
    induces a commuting square of $E(BT)-E(BT)$-bimodule spectra:
    \begin{equation}\label{eq:sts-to-st-commuting}
        \begin{tikzcd}
            \EBS_{i,j} \ar[rrr, "\mu_{iji}"] \ar[d, "\mu_{jij}"] &&& \EBS_{\underline{iji}} \ar[d, "\EBS_{\underline{ij}} \otimes_{E(BT)} m_i"] \\
            \EBS_{\underline{jij}} \ar[rrr, "m_j \otimes_{E(BT)} \EBS_{\underline{ij}}"] &&& \EBS_{\underline{ij}}
        \end{tikzcd}
    \end{equation}
    commutes.

    By the same argument, we have a commuting square $E(BT)-E(BT)$-bimodule spectra:
    \begin{equation}\label{eq:sts-to-ts-commuting}
        \begin{tikzcd}
            \EBS_{i,j} \ar[rrr, "\mu_{iji}"] \ar[d, "\mu_{jij}"] &&& \EBS_{\underline{iji}} \ar[d, "m_i \otimes_{E(BT)} \EBS_{\underline{ji}}"] \\
            \EBS_{\underline{jij}} \ar[rrr, "\EBS_{\underline{ji}} \otimes_{E(BT)} m_j"] &&& \EBS_{\underline{ji}}.
        \end{tikzcd}
    \end{equation}
\end{observation}

    




\section{$(\infty, 2)$-category of additive and stable $E$-valued Soergel bimodules}\label{sec:inf-2-SBim}

In this section we construct  the monoidal $(\infty, 2)$-categories $\SBim_E$ and $\Kbloc(\SBim_E)$ of $E$-valued additive and stable Soergel $(\infty, 2)$-categories, where $E$ a \emph{connective} $\EE_\infty$-ring spectrum.\footnote{See \cref{rem:non-connective-E} for the generalization to when $E$ is non-connective.}
Much of the argument here relies on the machinery developed in \cite{2024braided}. 

\subsection{The Morita category}\label{subsec:morita-cat}
Let us recall some results from \cite[\S 3,4]{2024braided}. 
Let $\st$ be the $\infty$-category of small idempotent-complete stable $\infty$-categories and $\Spectra$ be the stable $\infty$-category of spectra.

Throughout the rest of the subsection we fix $E \in \CAlg(\Spectra)$, that is, an $\EE_{\infty}$-ring spectrum.
If $\cC$ be a symmetric monoidal $\infty$-category, we denote by $\Cat[\cC]$ the $\infty$-category of $\cC$-enriched categories, and $\largecat[\cC]$ the $\infty$-category of large $\cC$-enriched categories.\footnote{We refer the reader to \cite[Appendix A.10]{2024braided} for a review of enriched $\infty$-categories.} 
We have $\Perf_E \coloneqq (\Mod_E(\Spectra))^{\mrc} \in \CAlg(\st)$ the stable $\infty$-category of compact $E$-module spectra.\footnote{By \cite[Lemma 3.5.7(2)]{2024braided} compact $E$-module spectra are retracts of iterated finite colimits of the regular $E$-module.} Let $\st_E \coloneqq \Mod_{\Perf_E}(\st)$. 
Now we define the relevant Morita category, which is a large $\infty$-category enriched in $\st_E$.
By \cite[Definition 4.4.4, Corollary 4.4.5]{2024braided}, with $\Monoid = \pt$ and $\mbbK = E$, we have a large symmetric monoidal $\st_E$-enriched $\infty$-category 
\begin{equation}
    \Morc_E \coloneqq \Moritac(\Mod_E) \in \CAlg(\largecat[\st_E])
\end{equation}
such that 
\begin{enumerate}
    \item There is a symmetric monoidal surjective-on-objects functor 
    $\Alg(\Mod_E) \to \Morc_E$. That is, we can label the objects of $\Morc_E$ by $E$-algebra spectra. 
    \item Given $A, B \in \Alg(\Mod_E)$, the $\st_E$-enriched hom 
    \begin{equation}
        \eHom_{\Morc_E}(A, B) \simeq {_A\BMod_B^{\mrc}(\Mod_E)}
    \end{equation}
    is the $\st_E$-enriched $\infty$-category of $A-B$-bimodule $E$-module spectra that are compact as right $B$-modules.
    \item Given $A, B, C \in \Alg(\Mod_E)$, the composition 
    \begin{equation}
        {_A\BMod_B^{\mrc}(\Mod_E)} \otimes {_B\BMod_C^{\mrc}(\Mod_E)}
        \to {_A\BMod_C^{\mrc}(\Mod_E)}
    \end{equation}
    is given by $- \otimes_B -$, i.e. tensoring over the middle algebra $B$.
    \item The symmetric monoidal structure $\boxtimes$ on $\Morc_E$ is given by  $- \otimes_E -$, i.e.  tensoring over $E$. It follows that $E$ is the unit of $\boxtimes$. 
\end{enumerate}
Just as in \cite{2024braided}, it will be convenient to work with a small full subcategory of $\Morc_E$:
\begin{definition}\label{def:PolyMorc}
   Let $\PolyMorc_E \in \CAlg(\Cat[\st_E])$ be the full symmetric monoidal subcategory of $\Morc_E$ containing the algebran $E(B\U(1))$.
\end{definition}
It follows that $\PolyMorc_E$ is the full subcategory of $\Morc_E$ containing the algebras $E(B\U(1))^{\otimes_E n} \simeq E(B\U(1)^n)$. 
\begin{observation}\label{obs:polymorc-description}
$\PolyMorc_E$ can be described as follows: it is a $\st_E$-enriched $\infty$-category with objects index by natural numbers. 
Given two natural numbers, the $\st_E$-enriched hom
\begin{equation}\label{eq:PolyMorc-homs}
    \eHom_{\PolyMorc_E}(n, m) = \, {_{E(B\U(1)^n)}\BMod_{E(B\U(1)^m)}^\mrc(\Mod_E)}
\end{equation}
is the $\st_E$-enriched $\infty$-category of $E(B\U(1)^n)-E(B\U(1)^m)$-bimodule $E$-module spectra that are compact as right $E(B\U(1)^m)$-modules.
Composition is given by tensoring over the middle algebra. 
Furthermore, the symmetric monoidal structure $\boxtimes$ is given 
by addition on objects $n \boxtimes m \coloneqq n+m$ and tensoring over $E$ on bimodules $M \boxtimes N \coloneqq M \otimes_E N$.
\end{observation}
Let us end this subsection with a simple yet important observation:
\begin{observation}\label{obs:permutation-bimodule}
    The $S_2$-action on the object $2 \in \PolyMorc$ is implemented by the permutation bimodule $$E(B\U(1)^2s_1) \in \eHom_{\PolyMorc_E}(2, 2) = \, {_{E(B\U(1)^2)}\BMod_{E(B\U(1)^2)}^\mrc(\Mod_E)}.$$
\end{observation}

\subsection{Additive and stable $E$-valued Soergel $(\infty,2)$-categories}\label{E-valed-SBim}
In this subsection we define the monoidal locally additive $(\infty, 2)$-category $\SBim_E$ of $E$-valued Soergel bimodules, as well as the locally stable analogue $\Kbloc(\SBim_E)$.
We will do this by defining a generating set of data, using the machinery of factorization systems of higher and enriched categories developed in \cite[\S 5, \S 6, Appendix B]{2024braided}. 

We start by defining the generating category $\Gen \in \Alg(\cat) \subset \Alg(\Cat[\cat])$. 
Let $\pt$ be the terminal category, $c_0 = S^0$ be the category with two objects and no nontrivial morphisms, and $\Delta^1$ be the arrow category.
\begin{definition}
    Let $\End^\mbbZ$ be the pushout in $\cat$:
    \begin{equation}
        \begin{tikzcd}
            (c_0)^{\sqcup \mbbZ} \ar[r] \ar[d] & (\Delta^1)^{\sqcup \mbbZ} \ar[d] \\ 
        \pt \ar[r]& \End^{\mbbZ}.
        \end{tikzcd}
    \end{equation}
\end{definition}
Intuitively, $\End^{\mbbZ}$ is the walking category with one object and $\mbbZ$ many endomorphisms. 
Let $\Free_{\EE_1} \colon \cat \to \Alg(\cat)$ be the free functor. 
\begin{definition}\label{def:gen}
   Let $\Gen$ be the pushout 
   \begin{equation}
    \begin{tikzcd}
        \Free(\pt) \ar[r, "* \mapsto * \boxtimes *"] \ar[d] & \Free(\pt) \ar[d] \\ 
        \Free(\End^\mbbZ) \ar[r] &  \Gen
    \end{tikzcd}
\end{equation}
in $\Alg(\cat)$.
\end{definition}
\begin{notation}
    We typically will use $\boxtimes$ to denote the monoidal tensor products.
\end{notation}
\begin{observation}\label{obs:mapping-out-of-Gen}
    By construction, $\Gen$ satisfies the following universal property:
    $\Gen$ is the universal monoidal $\infty$-category with a distinguished element $*$ together with $\mbbN$ many distinguished endomorphisms $f_k \in \Hom_{\Gen}(* \boxtimes *, * \boxtimes *)$.
    Let $\cD$ be a monoidal $(\infty, 1)$-category, then an $\EE_1$-map $G \colon \Gen \to \cD$ is equivalent to specifying an object $c = G(*) \in \cD$ together with $\mbbZ$ number of endomorphisms $G(f_k) \in \Hom_{\cD}(c \boxtimes c, c \boxtimes c)$. Note that this also holds when $\cD$ is an $(\infty, 2)$-category  via the fully faithful embedding $\Alg(\cat) \subset \Alg(\Cat[\cat])$.
\end{observation}

\begin{observation}\label{obs:describe-Gen}
   $\Gen$ has $\mbbN$ many objects, with $n$ labeling the $n$-th tensor product  $*^{\boxtimes n}$. Furthermore, $\eHom_{\Gen}(n,m) = \emptyset$ if $n \neq m$, and every morphism in $\eHom_{\Gen}(n,n)$ is a composition of morphisms of the form  $$f_{n,i,k} \coloneqq \id_{i-1} \boxtimes f_k \boxtimes \id_{n-i-1}$$ for $1 \leq i \leq n-1$ and $k \in \mbbZ$.
\end{observation}

Throughout the rest of the subsection, we fix $E \in \CAlg(\ConnSpectra)$ a \emph{connective} $\mbbE_\infty$-ring spectrum. Let us view $\PolyMorc_E$ defined in \cref{def:PolyMorc} as an object in $\Alg(\Cat[\cat])$. 
\begin{proposition}
    There exists a monoidal functor $F \colon \Gen \to \PolyMorc_E$ taking $*$ to $E(B\U(1))$ and the $k$-th endomorphism $f_k$ to the $k$-shifted Bott-Samelson bimodule $\EBS_1(k)$ associated to the simple transposition $s_1 = (1,2)$,
    which we view as an $E(B\U(1)^{2})-E(B\U(1)^{2})$ bimodule. 
\end{proposition}
\begin{proof}
    Let $T = \U(1)^{2}$.
    By \cref{obs:mapping-out-of-Gen} and \cref{obs:polymorc-description}, it suffices to check that $\EBS_1(k)$
    are indeed compact as right $E(BT)$-bimodules.
    Since suspensions of compact modules are still compact, it suffices to show that $\EBS_1$ is a compact $E(BT)$-bimodule. By \cref{prop:E-two-fiber-seq}, we have a fiber sequence 
    \begin{equation}
        E(BT)(2) \to \EBS_1 \to E(BTs_1).
    \end{equation}
    Note that as right $E(BT)$-module we have an isomorphism  $E(BTs_1) \simeq E(BT)$. Now the result follows from the fact that 
    $E(BT)$ and $E(BT)(2)$ are compact and 
    compactness is closed under extensions.
\end{proof}

\begin{observation}\label{obs:describe-F}
   Fix $n \in \mbbN$. Then $F$ takes the object $n = *^{\boxtimes n} \in \Gen$ to $E(B\U(1)^{n})$.
   Furthermore, suppose we have
   $1 \leq i \leq n-1$, and $k \in \mbbZ$. By \cref{obs:describe-Gen}, there is an endomorphism $f_{n,i,k} \coloneqq \id_{i-1} \boxtimes f_k \boxtimes \id_{n-i-1} \in \eHom_{\Gen}(n,n)$. $F$ takes  $f_{n,i,k}$ to the elementary Bott-Samelson
   \begin{equation}
   \EBS_{i}(k) \simeq E(B\U(1)^{i-1}) \otimes_E \EBS_{(1,2)}(k) \otimes_E E(B\U(1)^{n-i-1}) 
   \end{equation}
   in $\eHom_{\PolyMorc_E}(n,n) = \, {_{E(B\U(1)^n)}\BMod_{E(B\U(1)^n)}^\mrc(\Mod_E)}$. 
   Furthermore, given $\underbf{i} = (i_1, \cdots, i_m)$, by \cref{obs:tensor-over-EBT}, $F$ maps the composition 
   $ f_{n, i_1, k_1} \circ f_{n, i_2, k_2} \cdots \circ f_{n, i_m, k_m}$ to the Bott-Samelson $\EBS_{\underbf{i}}(k)$ where $k = \sum_{j} k_j$.
\end{observation}
Let $\add$ be the $\infty$-category of small idempotent-complete additive $\infty$-categories,\footnote{See \cite[\S 3.3]{2024braided} for a detailed discussion of additive $\infty$-categories.}
and $\ConnSpectra$ be the additive full subcategory of $\Sp$ consisting of connective spectra.
We have $\CProj_E \coloneqq (\Mod_E(\ConnSpectra))^{\cp} \in \CAlg(\add)$ the additive $\infty$-category of compact projective $E$-module spectra.\footnote{By \cite[Lemma 3.5.7(1)]{2024braided}, compact projective $E$-module spectra are  retracts of finite direct sums of the regular $E$-module.} Lastly we define $\add_E$ to be $\Mod_{\CProj_E}(\add)$.

\begin{proposition}\label{prop:a-bunch-of-left-adjoints}
    There are presentably symmetric monoidal functors 
\begin{equation}
    \begin{tikzcd}
         & \add_E \ar[r, "\Kb_E"] & \st_E \\ 
      \cat \ar[r] & \add \ar[r, "\Kb"] \ar[u, "-\otimes \CProj_E"] & \st \ar[u, "-\otimes \Perf_E"] 
    \end{tikzcd}
\end{equation}
whose right adjoints are the evident forgetful functors.
\end{proposition}
\begin{proof}
    See \cite[Eq. (6.1) and Proposition 3.4.5]{2024braided}.
\end{proof}
We will abuse notation and drop the $E$ in $\Kb_E$ when it is clear from context.
\begin{remark}\label{rem:Kb-generalizes-chain-complex}
    If $\cA$ is an $1$-category, then  $\Kb(\cA)$ is the $\infty$-categorical version of the chain homotopy category of $\cA$ by \cite[Corollary 3.4.10]{2024braided}. Formally, it is the dg nerve of the dg category of bounded chain complexes in $\cA$ (\cite[Definition 3.4.1]{2024braided}). Therefore $\Kb(\cA)$, where $\cA$ is an additive $(\infty, 1)$-category, is the generalization of the chain homotopy category for $(\infty, 1)$-categories.
\end{remark}
\begin{remark}
    The functors $\Kb_E$ and $\Kb$ commutes with forgetting the $E$ action. That is, we have a commutative diagram:
    \begin{equation}
        \begin{tikzcd}
            \add_E \ar[r, "\Kb_E"] \ar[d] & \st_E \ar[d]\\
            \add \ar[r, "\Kb"] & \st,
        \end{tikzcd}
    \end{equation}
    where the vertical maps are the evident forgetful functors.
\end{remark}
The composite $\cat \to \add \to \add_E$ induces a symmetric monoidal left adjoint 
\begin{equation}
    (-)^{\add_E} \colon \Cat[\cat] \to \Cat[\add_E].
\end{equation}
It follows that the functor $F \colon \Gen \to \PolyMorc_E$ in $\Alg(\Cat[\cat])$ induces a map 
\begin{equation}
    (F)^{\add_E} \colon (\Gen)^{\add_E} \to \PolyMorc_E 
\end{equation}
in $\Alg(\Cat[\add_E])$. 

We now define the $(\infty,2)$-category of $E$-valued Soergel bimodules via factorization systems.\footnote{We refer the reader to \cite[Appendix B]{2024braided} for an introduction to factorization systems.}
\begin{definition}\label{def:fully-faithful-and-dominant}
   Let $F \colon \cC \to \cD$ be a functor between $(\infty, 1)$-categories. $F$ is \emph{fully faithful} if for every $c_1, c_2 \in \cC$, the induced map of spaces
   \begin{equation}
    \Hom_{\cC}(c_1, c_2) \to \Hom_{\cD}(Fc_1, Fc_2)
   \end{equation}
   is an equivalence. $F$ is \emph{dominant} if every object in $\cD$ is a retract of $Fc$ for some $c \in \cC$.
\end{definition}
By \cite[Proposition 6.2.2]{2024braided}, the (dominant, fully faithful) functors define factorization system on $\add_E$. 
\begin{definition}
    Let $F \colon \cC \to \cD$ is a functor between $\add_E$-enriched categories. 
    $F$ is \emph{surjective on objects} if every object in $\cD$ is of the form $Fc$ for some $c \in \cC$. $F$ is \emph{dominant on $1$-morphism} if it is hom-wise dominant, that is, for every $c_1, c_2 \in \cC$, the induced map  
    $\eHom_{\cC}(c_1, c_2) \to \eHom_{\cD}(c_1, c_2)$ in $\add_E$ is dominant. Similarly, $F$ is \emph{faithful} if it is hom-wise fully faithful.
\end{definition}
By \cite[Corollary 6.2.4]{2024braided}, the (surjective on objects and dominant on $1$-morphisms, faithful) functors define a
factorization system on $\Cat[\add_E]$ and $\Alg(\Cat[\add_E])$.
\begin{definition}\label{def:definition-of-SBim_E}
    Let the monoidal $\add_E$-enriched $(\infty, 2)$-category of \emph{additive $E$-valued Soergel $(\infty, 2)$-category} $\SBim_E \in \Alg(\Cat[\add_E])$ be the unique factorization 
    \begin{equation}\label{eq:factorization-of-SBim}
        \begin{tikzcd}
            (\Gen)^{\add_E} \arrow[dr, "\substack{\mathrm{surjective-on-objects}\\\mathrm{-and-dominant-on-1-morphisms}}"'] \arrow[rr, "(F)^{\add_E}"] && \PolyMorc_E \\
            &\SBim_E \arrow[ur, "\mathrm{faithful}"']
            \end{tikzcd}
    \end{equation}
    of $(F)^{\add_E}$ with respect to the (surjective on objects and dominant on $1$-morphisms, faithful) factorization system.
\end{definition}
We would like to describe the objects and hom-categories of $\SBim_E$. First we need the following lemma:
\begin{lemma}\label{lem:dominant}
   Let $F \colon \cC \to \cD$  be a functor of $(\infty, 1)$-categories, where $\cC \in \cat$ and $\cD\in \add_E$. Then the following are equivalent:
   \begin{enumerate}
    \item Its adjunct $(\cC)^{\add_E} \to \cD$ is dominant. 
    \item Every object of $\cD$ is a retract of (finite) direct sums of objects in the image of $F$.
   \end{enumerate}
\end{lemma}
\begin{proof}
    The argument is the same as \cite[Lemma 6.3.1]{2024braided}, just without the $\mbbZ$ action.
\end{proof}
Applying \cref{lem:dominant} to $(\Gen)^{\add_E} \to \SBim_E$ in \eqref{eq:factorization-of-SBim}, we get 
\begin{corollary}\label{cor:Gen-to-SBim}
    Consider the map $\Gen \to \SBim_E$ adjunct to $(\Gen)^{\add_E} \to \SBim_E$ in \eqref{eq:factorization-of-SBim}; it is surjective on objects and every $1$-morphism in $\SBim_E$ is a retract of direct sums of $1$-morphisms in the image of $1$-morphisms in $\Gen$.
\end{corollary}
\begin{proposition}\label{prop:spell-out-SBim_E}
The objects of $\SBim_E$ are indexed by natural numbers $\mbbN$. 
Given $n , m \in \mbbN$; then  $\eHom_{\SBim_E}(n, m) = 0$ if $n \neq m$. Moreover,  $$\eHom_{\SBim_E}(n,n) \subset \, {_{E(B\U(1)^{n})}\BMod^\mrc_{E(B\U(1)^{n})}(\Mod_E)}$$  is the full subcategory consisting of retracts of direct sums of shifted $\U(n)$ Bott-Samelson bimodules
$\EBS_{\underbf{i}}(k)$.
\end{proposition}
\begin{proof}
    By construction, $\SBim_E$ has the same objects as $(\Gen)^{\add_E}$, which has the same objects as $\Gen$ as $(-)^{\add_E}$ is a hom-wise construction. Since $\Gen$ has $\mbbN$ many objects \cref{obs:describe-Gen}, it follows that $\SBim_E$ also has $\mbbN$ many objects.

    Fix $n, m \in \mbbN$; by \eqref{eq:factorization-of-SBim} we have a factorization
    \begin{equation}
        \begin{tikzcd}
            \eHom_{(Gen)^{\add_E}}(n,m) \ar[rr, "{F(n,m)}"] \ar[rd, "\mathrm{dominant}"']
            && \eHom_{\PolyMorc_E}(n,m). \\ 
            & \eHom_{\SBim_E}(n,m) \ar[ru, "\mathrm{fully-faithful}"']
            &
        \end{tikzcd}   
    \end{equation}
    By \cref{lem:dominant}, $\eHom_{\SBim_E}(n,m)$ is the full subcategory of $\eHom_{\PolyMorc_E}(n,m)$ consisting of retracts of 
    direct sums of $1$-morphisms in the image of $\Gen$. 

    Now we use  \cref{obs:polymorc-description} to identify $\eHom_{\PolyMorc_E}(n,m)$ with $_{E(B\U(1)^{n})}\BMod^\mrc_{E(B\U(1)^{n})}(\Mod_E)$. 
    Suppose $n \neq m$, since  $\eHom_{\Gen}(n,m) = \emptyset$ by \cref{obs:describe-Gen}, we see that  $\eHom_{\SBim_E}(n,m)$ is the $0$ category. On the other hand, 
    by \cref{obs:describe-F} we see that the image of $\eHom_{\Gen}(n,n)$ are precisely shifts of $\U(n)$ Bott-Samelson bimodules $\EBS_{\underbf{i}}(k)$. Therefore 
    $\eHom_{\SBim_E}(n,n)$ is the full subcategory consisting of retracts of direct sums of shifted Bott-Samelson bimodules.
\end{proof}
\begin{notation}\label{nota:SBim_E-n}
    We will often denote $\eHom_{\SBim_E}(n,n)$ as $\SBim_E(n)$.
\end{notation}
Now we define the stable Soergel $(\infty, 2)$-category.
Recall from \cref{prop:a-bunch-of-left-adjoints} that there is a presentably symmetric monoidal left adjoint $\Kb \colon \add_E \to \st_E$ to the evident forgetful functor. This induces a symmetric monoidal left adjoint 
\begin{equation}
    \Kbloc \colon \Cat[\add_E] \to \Cat[\st_E].
\end{equation}
\begin{definition}
    We call the monoidal $\st_E$-enriched $(\infty, 2)$-category $\Kbloc(\SBim_E)$  the \emph{$E$-valued stable Soergel $(\infty, 2)$-category}.
\end{definition}
\begin{observation}\label{obs:describe-KblocSBim_E}
   By construction, $\Kbloc(\SBim_E)$ has $\mbbN$ many objects, with 
   \begin{equation}
    \eHom_{\Kbloc(\SBim_E)}(n,m) = \begin{cases}
        0 & n \neq m \\
        \Kb(\SBim_E(n)) & n = m.
    \end{cases}
   \end{equation}
\end{observation}
\begin{remark}
   By \cref{rem:Kb-generalizes-chain-complex}, $\Kb(\SBim_E(n))$ is the $\infty$-categorical analogue of the chain homotopy category of Soergel bimodules. 
\end{remark}
By adjunction, the map $\SBim_E \to \PolyMorc_E$ in $\Alg(\Cat[\add_E])$ induces a map 
\begin{equation}\label{eq:fiber-functor}
    H_{\loc} \colon \Kbloc(\SBim_E) \to \PolyMorc_E
\end{equation}
in $ \Alg(\Cat[\st_E])$. This is our \emph{fiber functor}.
\begin{remark}\label{rem:suspensions}
    Fix $n \in \mbbN$;
    there are two type of ``suspensions'' in $\Kb(\SBim_E(n)) = 
    \eHom_{\Kbloc(\SBim_E)}(n,n)$. There is the standard suspension for the stable $\infty$-category $\Kb(\SBim_E(n))$, which we denote by $\Sigma$. It corresponds to the shifts in chain complex degrees.
    On the other hand, there are the ``grading shifts'' $(k) \colon \SBim_E \xrightarrow{\simeq} \SBim_E$ of Soergel bimodules, which extend to automorphisms on $\Kbloc(\SBim_E)(n)$. While $\Sigma$ and $(-1)$ disagree on $\Kbloc$, 
    the functor $H_{\loc}$ takes both to the standard suspension 
    in $\left(_{E(B\U(1)^n)}\BMod_{E(B\U(1)^n)}^\mrc(\Mod_E)\right)  = \eHom_{\PolyMorc_E}(n,n)$. 
\end{remark}


\section{Braiding on $\Kbloc(\SBim_E)$}\label{sec:braiding-on-Kbloc}
In this section we construct a braiding on $\Kbloc(\SBim_E)$, where $E$ is a connective $\EE_\infty$-ring spectrum equipped with a complex $\EE_\infty$-orientation $f_E \colon \MU \to E$.

\subsection{From braiding to prebraiding to Rouquier complex}\label{subsec:braiding-to-Rouquier}
We start by recalling the definition of a braiding as well as the parallel notion of a prebraiding. 
\begin{definition}\label{def:braiding}
    Let $\cC$ be a symmetric monoidal $\infty$-category and $c \in \Alg(\cC) = \Alg_{\EE_1}(\cC)$. Then the \emph{space of braidings} on $c$ is 
    \begin{equation}
    \Braid_{\cC}(c) \coloneqq \Alg_{\EE_2}(\cC) \times_{\Alg_{\EE_1}(\cC)} \{c\}.
    \end{equation}
\end{definition}
Note that this is a space as the forgetful map $\Alg_{\EE_2}(\cC) \to \Alg_{\EE_1}(\cC)$ is conservative.
Intuitively a braiding on $c$ is a lift of the $\EE_1$-algebra structure on $c$ to an $\EE_2$-algebra structure.

We also have the the notion of prebraidings on monoidal functors between $1$-categories over a symmetric monoidal $1$-category: 
\begin{definition}\label{def:prebraiding}{\cite[Definition 2.4.1]{2024braided}}
Let $\cA$ and $\cB$ be monoidal $1$-categories, with monoidal product denoted by
$\boxtimes$ in both cases and with associators $b_{x,y,z}$ in $\cB$. A \emph{prebraiding} $\beta$ on a monoidal functor $F\colon \cA \to \cB$ consists of the data
of isomorphisms
\[ F(x)\boxtimes F(y) \xrightarrow{\beta_{x,y}} F(y) \boxtimes F(x)\qquad
\forall x,y\in \cA \] that form a natural transformation $\boxtimes\circ (F\times
F) \Rightarrow \boxtimes^{\mathrm{op}}\circ (F\times F)$ and satisfy the
following two \emph{hexagon axioms} 
 for all $x,y,z\in \cA$:
    \begin{equation}
    \label{eq:fakehexagon-top}
        \begin{tikzcd}
            [scale cd=.75]
       (F(x)\boxtimes F(y))\boxtimes F(z) \arrow[d,"b"] \arrow[r,"\beta_{x,y} \boxtimes \id"]
       & (F(y)\boxtimes F(x))\boxtimes F(z) \arrow[r,"b"] 
       &  F(y)\boxtimes (F(x)\boxtimes F(z)) \arrow[r,"\id\boxtimes \beta_{x,z}"] 
       &   F(y)\boxtimes (F(z)\boxtimes F(x)) \arrow[d,"b^{-1}"]
       \\ 
       F(x)\boxtimes (F(y)\boxtimes F(z))\arrow[r,"\simeq"]
        & F(x)\boxtimes F(y\boxtimes z) \arrow[r,"\beta_{x,y\boxtimes z}"]
        & F(y \boxtimes z)\boxtimes F(x) \arrow[r,"\simeq"] 
        & (F(y)\boxtimes F(z)) \boxtimes F(x)
        \end{tikzcd}
    \end{equation}
    \begin{equation}
        \label{eq:fakehexagon-bottom}
            \begin{tikzcd}
                [scale cd=.75]
       F(x)\boxtimes (F(y)\boxtimes F(z)) \arrow[d,"b^{-1}"] \arrow[r,"\id\boxtimes \beta_{y,z}"]
       & F(x)\boxtimes (F(z)\boxtimes F(y)) \arrow[r,"b^{-1}"] 
       &  (F(x)\boxtimes F(z))\boxtimes F(y) \arrow[r," \beta_{x,z}\boxtimes\id"] 
       &   (F(z)\boxtimes F(x))\boxtimes F(y) \arrow[d,"b"]
       \\ 
       (F(x)\boxtimes F(y))\boxtimes F(z)\arrow[r,"\simeq"]
        & (F(x\boxtimes y))\boxtimes F(z) \arrow[r,"\beta_{x\boxtimes y,z}"]
        & F(z) \boxtimes F(x\boxtimes y) \arrow[r,"\simeq"] 
        & F(z)\boxtimes (F(x) \boxtimes F(y))
        \end{tikzcd} 
     \end{equation}
     where the isomorphisms $\simeq$ are part of the data of $F$.
     We denote the \emph{set of prebraidings} on $F$ by $\PreBraid(F)$.
\end{definition}
\begin{observation}\label{obs:prebraiding-on-identity}\cite[Corollary 2.4.2]{2024braided}
    Let $\cA$ be a monoidal $1$-category; then a pre-braiding on the identity functor $\id \colon \cA \to \cA$ is equivalent to a braided monoidal structure on $\cA$ (in the sense of \cite[Definition 8.1.1]{EGNO}) that extends its monoidal structure. Furthermore, since a braided monoidal structure on $\cA$ is equivalent to a $\EE_2$-algebra structure on $\cA$ (see \cite[Remark 7.7.7]{2024braided}),
    we have an isomorphism 
    \begin{equation}
        \PreBraid(\id \colon \cA \to \cA) \simeq \Braid_{\Cat_{(1,1)}}(\cA).
    \end{equation}
    See \cite[Example 8.1.2]{2024braided} for a detailed discussion.
\end{observation}
\begin{remark}
    See \cite[Notation 8.1.1]{2024braided} for the 
    $\infty$-categorical generalization of prebraiding.
\end{remark}
We also need the notion of prebraiding over a base braided monoidal $1$-category.
\begin{definition}\label{def:prebraiding-over-base}\cite[Definition 2.4.5]{2024braided}
   Suppose we have monoidal categories $\cC_1, \cC_2$ and a braided monoidal category $\cD$. Furthermore, suppose we have monoidal functors $F \colon \cC_1 \to \cC_2$ and $g \colon \cC_2 \to \cD$. We can consider $\cC_1$ as a monoidal category over $\cD$ by the composite $f \coloneqq g \circ F$.

   A \emph{prebraiding on $F$ over $\cD$ } is a prebraiding on $F$ together with the condition that g maps the prebraiding isomorphisms $\beta$ to the braiding isomorphisms in $\cD$. We denote by $\PreBraid_{/\cD}(F)$ the set of prebraidings on $F$ over $\cD$.
\end{definition}
\begin{observation}\cite[Corollary 2.4.7]{2024braided}
    Let $\cA$ be a monoidal category, $\cD$ a braided monoidal category, and $g \colon \cA \to \cD$ be a monoidal functor. Then a prebraiding on the identity functor $\id \colon \cA \to \cA$ over $\cD$ is a braided monoidal structure on $\cA$ such that $g$ is braided monoidal in the sense of \cite[Definition 8.1.7]{EGNO}. As braided monoidal functors are the same as $\EE_2$ algebra maps, 
    we have an equivalence 
    \begin{equation}\label{eq:prebraid-on-iden-to-braid-rel}
        \PreBraid_{/\cD}(\id \colon \cA \to \cA) \simeq \Braid_{(\Cat_{(1,1)})_{/\cD}}(\cA).
    \end{equation}
\end{observation}

 


Now let us recall our setup. Let $E$ be a connective $\EE_\infty$-algebra. We have $\Kbloc(\SBim_E)$ together with a fiber functor $$H_{\loc} \colon \Kbloc(\SBim_E) \to \PolyMorc_E$$ in $ \Alg_{\EE_1}(\Cat[\st_E])$.
We are interested in lifting the $\EE_1$-algebra structures on both $\Kbloc(\SBim_E)$ and $H_{\loc}$ to $\EE_2$-algebra structures. Equivalently, we want a braiding on $\Kbloc(\SBim_E)$, viewed as an object in the over-category $\Alg_{\EE_1}(\Cat[\st_E])_{/\PolyMorc_E} \simeq \Alg_{\EE_1}(\Cat[\st_E]_{/\PolyMorc_E})$ (\cite[A.8.6]{2024braided}) via $H_\loc$.

We now reduce a problem about braiding on an $(\infty, 2)$-category to a problem about prebraiding on a $1$-category, namely its homotopy $1$-category.
\begin{definition}\label{def:homotopy-1-functor}
    Let $h_0 \colon \cat \to \Set$ denote the symmetric monoidal functor that takes an $\infty$-category to the set of isomorphism classes of objects. 
    This induces $h_1 \colon \Cat[\cat] \to \Cat_{(1,1)} = \Cat[\Set]$, 
    which takes an $(\infty, 2)$-category $\cC$ to a $1$-category $h_1 \cC$ who has the same objects and whose $1$-morphisms are isomorphism classes of $1$-morphisms in $\cC$. We call $h_1\cC$ the \emph{homotopy $1$-category} of $\cC$.\footnote{We refer the reader to \cite[\S 5]{2024braided} for detailed discussions about the functors $h_1$, and more generally the homotopy $n$-category functors $h_n$.} 
\end{definition}
We have a sequence of maps 
\begin{equation}\label{eq:reducing-to-prebraid}
    \begin{aligned}
        \Braid_{\Cat[\st_E]_{/\PolyMorc_E}}(\Kbloc(\SBim_E)) &\to 
        \Braid_{{\Cat_{(1,1)}}_{/h_1\PolyMorc_E}}(h_1 \Kbloc(\SBim_E)) 
         \\&
         \simeq \PreBraid_{/h_1\PolyMorc_E}(\id \colon h_1 \Kbloc(\SBim_E) \to h_1 \Kbloc(\SBim_E))
         \\& \to 
        \PreBraid_{/h_1\PolyMorc_E}(h_1 \Gen \to h_1 \Kbloc(\SBim_E)),
    \end{aligned}
\end{equation}
where $h_1 \Gen \to \Kbloc(\SBim_E)$ comes from the composite $\Gen \to \SBim_E \to \Kbloc(\SBim_E)$.
Now we use the key technical result of \cite{2024braided}:
\begin{proposition}\label{prop:reducing-to-prebraid}
   The composite map
   $$ \Braid_{\Cat[\st_E]_{/\PolyMorc_E}}(\Kbloc(\SBim_E)) \to 
   \PreBraid_{/h_1\PolyMorc_E}(h_1 \Gen \to h_1 \Kbloc(\SBim_E))$$ is an isomorphism.  
\end{proposition}
\begin{proof}
    The result follows from  \cite[Theorem 8.2.1]{2024braided}
    with $\cC = \SBim_E$, $\cD = \PolyMorc_E$, and $\cB = \Gen$, together with the described maps between them.\footnote{\cite[Theorem 8.2.1]{2024braided} is stated for $E = Hk$ with a $\mbbZ$-grading. However, the statement works more generally for any $E$ and with or without grading, see \cite[Footnote 37]{2024braided}. In our case we take $\mbbK$ to be $E$ and $\Monoid = \pt$.} 
    It remains to check that they satisfy the conditions of \cite[Theorem 8.2.1]{2024braided}. The map $\SBim_E \to \PolyMorc_E$ is faithful by definition \eqref{eq:factorization-of-SBim}, and the condition on $\Gen \to \SBim_E$ is proven in \cref{cor:Gen-to-SBim}. Note that there are no $\mbbZ$-shifts as we don't have a $\mbbZ$-grading.
\end{proof}

It remains to define a prebraiding on $h_1 \Gen \to h_1 \SBim_E$ over $h_1\PolyMorc_E$. We would like to simplify the data needed to define a prebraiding from $h_1 \Gen$. 
\begin{lemma}\label{lem:prebraid-to-slide-maps}
    Let $\cC$ be a monoidal $1$-category and $F \colon h_1\Gen \to \cC$ be a monoidal map given by taking $* \mapsto c$ and $f_k \mapsto g_k \colon c \boxtimes c \to c \boxtimes c$.
    The set of prebraiding $\PreBraid(F \colon h_1 \Gen \to \cC)$ isomorphic to the set of morphism $\beta \colon c \boxtimes c \to c \boxtimes c$ satisfying the following:
    \begin{enumerate}
        \item $\beta$ is an equivalence.
        \item For every $k \in \mbbZ$, the following diagrams commute:\footnote{For the rest of the section, to simplify notation, we omit the parentheseses and associators.}
    \begin{equation}\label{eq:slide-map-top}
        \begin{tikzcd}
            c \boxtimes c \boxtimes c \ar[rr, "\id_c \boxtimes g_k"] \ar[d, "\beta \boxtimes  \id_c"] && c \boxtimes c \boxtimes c \ar[d, "\beta \boxtimes \id_c"] \\ 
            c \boxtimes c \boxtimes c \ar[d, "\id_c \boxtimes \beta"] && c \boxtimes c \boxtimes c \ar[d, "\id_c \boxtimes \beta"]  \\ 
            c \boxtimes c \boxtimes c \ar[rr, "g_k \boxtimes  \id_c"] &&  c \boxtimes c \boxtimes c
        \end{tikzcd}
    \end{equation}
        \begin{equation}\label{eq:slide-map-bottom}
\begin{tikzcd}
    c \boxtimes c \boxtimes c \ar[rr, "g_k \boxtimes \id_c"] \ar[d, "\id_c \boxtimes \beta"] && c \boxtimes c \boxtimes c \ar[d, "\id_c \boxtimes \beta"] \\ 
    c \boxtimes c \boxtimes c \ar[d, "\beta \boxtimes \id_c"] && c \boxtimes c \boxtimes c \ar[d, "\beta \boxtimes \id_c"]  \\ 
    c \boxtimes c \boxtimes c \ar[rr, "\id_c \boxtimes g_k"] &&  c \boxtimes c \boxtimes c.
\end{tikzcd}
    \end{equation}
    \end{enumerate}
\end{lemma}
Intuitively, the diagrams \eqref{eq:slide-map-top} \eqref{eq:slide-map-bottom} are saying that the endomorphisms $f_k$ slide across the braiding given by $\beta$.
\begin{proof}
    First we construct $\beta$ from a prebraiding and show that it satisfies the required conditions.
    Suppose we have a prebraiding natural transform $\gamma \colon \boxtimes \circ (F \times F) \simeq \boxtimes^{\op} \circ (F \times F)$.
    We define $\beta$ to be $\gamma_{1,1} \colon c \boxtimes c \mapsto c \boxtimes c$. It is clearly invertible. Now we would like to show that $\beta$ satisfies condition (2). By the top hexagon axiom \cref{eq:fakehexagon-top}, the braiding $\gamma_{1, 2} \colon c \boxtimes c \boxtimes c \to c \boxtimes c \boxtimes c$ is given by the composite 
    \begin{equation}
        c \boxtimes c \boxtimes c \xrightarrow{\beta \boxtimes \id_c} 
        c \boxtimes c \boxtimes c \xrightarrow{\id_c \boxtimes \beta}
        c \boxtimes c \boxtimes c.
    \end{equation}
    Now the commutativity of \eqref{eq:slide-map-top} follows from the naturality of $\gamma$. The analogous argument shows that the commutativity \eqref{eq:slide-map-bottom} follows from the bottom Hexagon axiom \eqref{eq:fakehexagon-bottom} and naturality of $\gamma$.

    Now suppose we have $\beta \colon c \boxtimes c \xrightarrow{\simeq} c \boxtimes c$ satisfying \eqref{eq:slide-map-top} and \eqref{eq:slide-map-bottom}. We will now construct a prebraiding on $F \colon h_1\Gen \to \cC$. First we construct the natural isomorphism $\gamma \colon \boxtimes \circ (F \times F) \simeq \boxtimes^{\op} \circ (F \times F)$. 
    Given $n \in \mbbN$ and a simple transposition $s_i$ in $S_n$, we define $F(s_i)$ to be $$\id_{c^{\boxtimes i-1}} \boxtimes \beta \boxtimes \id_{c^{\boxtimes n-i-1}} \in \Hom_{\cC}(c^{\boxtimes n}, c^{\boxtimes n}).$$
    More generally, given $\underbf{i} = (i_1, \cdots, i_k)$, we define
    \begin{equation}
        F(\underbf{i}) \coloneqq F(s_{i_1}) \circ \cdots \circ F(s_{i_k}) \in 
        \Hom_{\cC}(c^{\boxtimes n}, c^{\boxtimes n}).
    \end{equation}
   Now we define the prebraiding:
   \begin{equation}\label{eq:cable-crossing}
    \gamma_{m,n} \coloneqq F(\underline{(s_n \cdots s_1) \cdots (s_{i+n-1} \cdots s_i) \cdots (s_{m+n-1} \cdots s_n)}) \in \Hom_{\cC}(c^{\boxtimes m + n}, c^{\boxtimes m + n}).
   \end{equation}
   See \cite[Figure 1]{2024braided} for pictures of such $\gamma$.
   
   Now we need to show that the assignments $\gamma_{m,n}$ are natural and that they satisfy the two Hexagon axioms. Both of the Hexagon axioms \eqref{eq:fakehexagon-top} \eqref{eq:fakehexagon-bottom} follow directly from the construction of $\gamma_{m,n}$ and the naturality of $\boxtimes$, which allow us to exchange the order of far away crossings.

   As for naturality, we need show that for any $m, n \in \mbbN$ and any $f \in \Hom_{\Gen}(n,n)$,\footnote{Recall that $\Gen$ only has endomorphisms by \cref{obs:describe-Gen}.} the morphism $f$ can slide across $\gamma_{m,n}$:
   \begin{equation}\label{eq:f-right-slide}
    \begin{tikzcd}
       c^{\boxtimes m} \boxtimes c^{\boxtimes n} \ar[d, "\id_{c^{\boxtimes m}} \boxtimes  F(f)"] \ar[rrr, "\gamma_{m,n}"]  
       &&& c^{\boxtimes n} \boxtimes c^{\boxtimes m} \ar[d, "F(f) \boxtimes  \id_{c^{\boxtimes n}}"] \\ 
       c^{\boxtimes m} \boxtimes c^{\boxtimes n} \ar[rrr, "\gamma_{m,n}"] 
       &&& c^{\boxtimes n} \boxtimes c^{\boxtimes m}.
    \end{tikzcd}
   \end{equation}
   Similarly, for any $f \in \Hom_{\Gen}(m,m)$, it can slide across $\gamma_{m,n}$:
   \begin{equation}\label{eq:f-left-slide}
    \begin{tikzcd}
       c^{\boxtimes m} \boxtimes c^{\boxtimes n} \ar[d, "F(f) \boxtimes \id_{c^{\boxtimes n}}"] \ar[rrr, "\gamma_{m,n}"]  
       &&& c^{\boxtimes n} \boxtimes c^{\boxtimes m} \ar[d, "\id_{c^{\boxtimes n}} \boxtimes F(f)"] \\ 
       c^{\boxtimes m} \boxtimes c^{\boxtimes n} \ar[rrr, "\gamma_{m,n}"] 
       &&& c^{\boxtimes n} \boxtimes c^{\boxtimes m}.
    \end{tikzcd}
   \end{equation}
   Let us prove the commutativity of \eqref{eq:f-right-slide}. By \cref{obs:describe-Gen}, it suffices to consider the case of $f$ being  $f_{n,i,k} \coloneqq \id_{i-1} \boxtimes f_k \boxtimes \id_{n-i-1} \in \Hom_{\Gen}(n,n)$ for some $1 \leq i \leq n-1$ and $k \in \mbbZ$. 
   Furthermore, using the functoriality of $\boxtimes$, we can reduce to the case when $m = 1$, $n = 2$, and $f = f_k$. In this case \eqref{eq:f-right-slide} unpacks to \eqref{eq:slide-map-top}.
   Similarly, we can reduce the commutativity of \eqref{eq:f-left-slide} to that of \eqref{eq:slide-map-bottom}.

   To finish the proof, we need to check that the two constructions above are inverse constructions. The case of  $\beta$ to prebraiding back to $\beta$ is clear. On the other hand, for any prebraiding $\gamma$, the hexagon axioms \eqref{eq:fakehexagon-top} \eqref{eq:fakehexagon-bottom} imply that $\gamma_{m,n}$ is of the form defined in \eqref{eq:cable-crossing}.
\end{proof}

\begin{remark}
    \cref{lem:prebraid-to-slide-maps} is a rigorous formalization of much of the reduction arguments given in \cite[\S 2.5]{2024braided}.
\end{remark}

We apply \cref{lem:prebraid-to-slide-maps} to $h_1 \Gen \to \h_1 \Kbloc(\SBim_E)$:
\begin{corollary}\label{cor:reduce-to-Rouquier-and-slide}
    The set of prebraidings $\PreBraid_{/h_1\PolyMorc_E}(h_1 \Gen \to h_1 \Kbloc(\SBim_E))$ is equivalent to the set of isomorphism classes of morphisms $R \in \eHom_{\Kbloc(\SBim_E)}(2,2)$ satisfying the following conditions:
    \begin{enumerate}
        \item The image $$H_\loc(R) \in \eHom_{\PolyMorc_E}(2,2) = \,{_{E(B\U(1)^2)}\BMod_{E(B\U(1)^2)}^\mrc(\Mod_E)}$$ is isomorphic to the permutation bimodule $E(B\U(1)^2s_1)$.
        \item $R$ is invertible. 
        \item 
        Let $R_1, R_2 \in \Hom_{\Kbloc(\SBim_E)}(3,3)$ denote $R \boxtimes \id_{E(BT)}$ and $\id_{E(BT)}  \boxtimes R$ respectively, then there exists isomorphisms
        \begin{equation}\label{eq:BS-slide-1}
            \EBS_1 \circ R_2 \circ R_1 \simeq 
            R_2 \circ R_1 \circ \EBS_2,
        \end{equation}
    and 
        \begin{equation}\label{eq:BS-slide-2}
            \EBS_{2} \circ R_1 \circ R_2 \simeq 
            R_1 \circ R_2 \circ \EBS_1.
        \end{equation}
        Here we use $\circ$ to represent the composition of $1$-morphisms in $\Hom_{\Kbloc(\SBim_E)}(3,3)$.
    \end{enumerate}
\end{corollary}
\begin{proof}
    By \cref{lem:prebraid-to-slide-maps}, a prebraiding $\gamma$ on $h_1 \Gen \to h_1 \Kbloc(\SBim_E)$ is 
    equivalent to the set of isomorphism classes of such $R$ satisfying conditions (2) and (3). Note that it suffices to check on $\EBS_i$ because the other generating endomorphisms are mapped to shifts of $\EBS_i$. 
    Furthermore, by \cref{obs:permutation-bimodule}, condition (1) is equivalent to the condition of $\gamma$ being a prebraiding  \emph{over} $h_1\PolyMorc$ in the sense of \cref{def:prebraiding-over-base}.
\end{proof}

\begin{remark}
    Note that we only ask for the isomorphisms in conditions (1) and (3) of \cref{cor:reduce-to-Rouquier-and-slide} to exist, as we are working on the level of homotopy $1$-categories.
\end{remark}

\subsection{Main theorem}\label{subsec:main-theorem}
In this subsection we construct a braiding on $\Kbloc(\SBim_E)$ by defining the Rouquier complex and show that it satisfies the conditions of \cref{cor:reduce-to-Rouquier-and-slide}.
\begin{definition}\label{def:rouquier-complex}
    Let $T = \U(1)^2$. We define the \emph{Rouquier complex} $$R \in \Kb(\SBim_E(2)) = \Hom_{\Kbloc(\SBim_E)}(2,2)$$ to be the fiber of  $$m_1(-2) \colon \EBS_1(-2) \to E(BT)(-2)$$ defined in \cref{prop:E-two-fiber-seq}.
\end{definition}

Now we show the Rouquier complex satisfies the conditions of \cref{cor:reduce-to-Rouquier-and-slide}. We start with condition (1):
\begin{proposition}\label{prop:condition-1}
   The functor $H_\loc \colon \Kbloc(\SBim_E) \to \PolyMorc_E$ takes the Rouquier complex $R$ to the permutation bimodule $E(BTs_1) \in \Hom_{\PolyMorc_E}(2,2)$.
\end{proposition}
\begin{proof}
    The functor $H_{\loc}(2,2) \colon \Hom_{\Kbloc(\SBim_E)}(2,2) \to \Hom_{\PolyMorc_E}(2,2)$ is exact.\footnote{Recall that a functor between stable $\infty$-categories is exact if it takes finite (co)limits to finite (co)limits.} Therefore it takes the Rouquier complex $R$ to the fiber of $$m_i(-2) \colon \EBS_i(-2) \to E(BT)(-2)$$ in $$\Hom_{\PolyMorc_E}(2,2) = \,{_{E(B\U(1)^2)}\BMod_{E(B\U(1)^2)}^\mrc(\Mod_E)}.$$ Now the result follows from the anti-diagonal fiber sequence in \cref{prop:E-two-fiber-seq}.
\end{proof}

Next we move on to condition (2) of \cref{cor:reduce-to-Rouquier-and-slide} that $R$ is invertible. First we define the inverse:
\begin{definition}
    We define the \emph{inverse Rouquier complex} $R'\in \Kb(\SBim_E(2))$ to be the cofiber of $\Deltas_1 \colon E(BT) \to \EBS_1$ defined in \cref{prop:E-two-fiber-seq}.
\end{definition}
We need the basic aspects of the theory of total fibers. We refer the reader to \cite[\S 3.4]{cubical} for the general theory.
\begin{definition}\label{def:total-fiber}
   Let $\cC$ be a stable $\infty$-category. Consider a commutative square
   \begin{equation}
    \begin{tikzcd}
        A \ar[r] \ar[d] & B \ar[d] \\ 
        C \ar[r] & D
    \end{tikzcd}
   \end{equation} in $\cC$. The \emph{total fiber} of the diagram is the fiber of the induced map $A \to B \times_D C$.
\end{definition}
\begin{lemma}\label{lem:total-fiber-inductive}\cite[Proposition 3.4.3]{cubical}
    Let $\cC$ be a stable $\infty$-category. Consider a commutative square
    \begin{equation}
     \begin{tikzcd}
         A \ar[r] \ar[d] & B \ar[d] \\ 
         C \ar[r] & D
     \end{tikzcd}
    \end{equation} in $\cC$.
    Then the total fiber of the square is equivalent to the fiber of 
    \begin{equation}
        \fib(A \to B) \to \fib(C \to D).
    \end{equation}
   Analogous, the total fiber of the square is also equivalent to the fiber of 
   \begin{equation}
    \fib(A \to C) \to \fib(B \to D).
   \end{equation} 
\end{lemma}
\begin{lemma}\label{lem:total-fiber-of-tensors}
    Let $\cC \in \Alg(\st)$ be a monoidal stable $\infty$-category with tensor product $\boxtimes$. Let 
    $A \to B \to C$ and $A' \to B' \to C'$ be fiber sequences in $\cC$. Then $A \boxtimes A'$ is the total fiber of 
    \begin{equation}
        \begin{tikzcd}
            B \boxtimes B' \ar[r] \ar[d] & B \boxtimes C' \ar[d]\\ 
            C \boxtimes B' \ar[r] & C \boxtimes C'.
        \end{tikzcd}
    \end{equation}
\end{lemma}
\begin{proof}
    Consider the following diagram:
    \begin{equation}
        \begin{tikzcd}
        A \boxtimes A' \ar[r] \ar[d]& A \boxtimes B' \ar[r] \ar[d] & A \boxtimes C' \ar[d] \\ 
        B \boxtimes A' \ar[r] \ar[d] &   B \boxtimes B' \ar[r] \ar[d] & B \boxtimes C' \ar[d]\\ 
        B \boxtimes A' \ar[r]&   C \boxtimes B' \ar[r] & C \boxtimes C'. 
        \end{tikzcd}
    \end{equation}
    Since $\boxtimes$ is exact in both variables, each row and column is a fiber sequence. In particular, the top row is a fiber sequence. Now the result follows from \cref{lem:total-fiber-inductive} as the last two columns are also fiber sequences.
\end{proof}
We also have a straightforward yet useful observation:
\begin{observation}\label{obs:eliminating-direct-sums}
    Let $\cC$ be a stable $\infty$-category. Suppose we have a commutative square of the form 
    \begin{equation}
        \begin{tikzcd}
            A \ar[r] \ar[d] & D \oplus B \ar[d, "{(\id, 0)}"] \\ 
            C \ar[r] & D \, .
        \end{tikzcd}
    \end{equation}
    Then the total fiber of the above square is equivalent to the total fiber of 
    \begin{equation}
        \begin{tikzcd}
            A \ar[r] \ar[d] & B\ar[d] \\ 
            C \ar[r]& 0\, ,
        \end{tikzcd}
    \end{equation}
    which is simply the fiber of $A \to B \oplus C$.
\end{observation}

Now we prove condition (2) of \cref{cor:reduce-to-Rouquier-and-slide}, that is, that $R$ and $R'$ are inverses of each other:
\begin{proposition}\label{prop:condition-2}
    We have equivalences 
    \begin{equation}
    R \circ R' \simeq E(BT) \simeq R' \circ R.        
    \end{equation}
    in $\Kb(\SBim_E(2))$. Here $\circ$ is denoting the monoidal composition product on $\Kb(\SBim_E(2)) = \Hom_{\Kbloc(\SBim_E)}(2, 2)$.
\end{proposition}

\begin{proof}
    Let us prove that $R \circ R' \simeq E(BT)$; the proof of $E(BT) \simeq R' \circ R$ is analogous. By \cref{lem:total-fiber-of-tensors}, $\Sigma R \circ R'$ is the
    total fiber of 
    \begin{equation}\label{eq:R-otimes-Rprime-diag-1}
        \begin{tikzcd}
            \EBS_1 \ar[rrr, "\EBS_1 \otimes_{E(BT)} \Deltas_1"] \ar[d, "m_1"] &&& \EBS_{\underline{11}}(-2) \ar[d, "m_1 \otimes_{E(BT)} \EBS_i(-2)"] \\ 
            E(BT) \ar[rrr] &&& \EBS_1(-2).
        \end{tikzcd}
    \end{equation}

    By \cref{prop:E-BS_s-splitting-1}, we have a splitting
    \begin{equation}
        \EBS_{\underline{11}} \simeq \EBS_1 \oplus \EBS_1(2)
    \end{equation}
    in $\SBim_E(2)$. Since the map $\SBim_E(2) \to \Kb(\SBim_E(2))$ is additive, this splitting also holds in $\Kb(\SBim_E(2))$. Replacing $\EBS_{\underline{11}}(-2)$ with $\EBS_1(-2) \oplus \EBS_1$, we can replace \eqref{eq:R-otimes-Rprime-diag-1} with 
    \begin{equation}\label{eq:R-otimes-Rprime-diag-2}
        \begin{tikzcd}
            \EBS_1 \ar[rrrr, "{(?,\, (\EBS_1 \otimes_{E(BT)} \Deltas_1) \circ \nabla_1^L)}"] \ar[d, "m_1"] &&&& \EBS_1(-2) \oplus \EBS_1 \ar[d, "{(\id, 0)}"] \\ 
            E(BT) \ar[rrrr] &&&& \EBS_1(-2).
        \end{tikzcd} 
    \end{equation}
    Note that we denote maps that does not need specification by $?$.
    By \cref{obs:eliminating-direct-sums}, the total fiber of \eqref{eq:R-otimes-Rprime-diag-2} is equivalent to the fiber of 
    \begin{equation}\label{eq:condition-2-simplify}
        \EBS_1 \xrightarrow{((\EBS_1 \otimes_{E(BT)}\Deltas_1) \circ \nabla_1^L,\, m_1)} \EBS_1 \oplus E(BT).
    \end{equation}
    Since $(\EBS_1 \otimes_{E(BT)}) \Deltas_1 \circ \nabla_1^L$ is an equivalence by \cref{lem:need-this-for-condition-2}, we see that the fiber of \eqref{eq:condition-2-simplify} is $\Sigma E(BT)$. Therefore we have an equivalence $ R \circ R' \simeq E(BT)$.
\end{proof}
Lastly, we prove condition (3) of \cref{cor:reduce-to-Rouquier-and-slide}:
\begin{proposition}\label{prop:condition-3}
    We have equivalences 
         \begin{equation}\label{eq:BS-slide-1-copy}
            \EBS_1 \circ R_2 \circ R_1 \simeq 
            R_2 \circ R_1 \circ \EBS_2,
        \end{equation}
    and 
        \begin{equation}\label{eq:BS-slide-2-copy}
            \EBS_2 \circ R_1 \circ R_2 \simeq 
            R_1 \circ R_2 \circ \EBS_1.
        \end{equation}
        in $\Kb(\SBim_E(3))$.
\end{proposition}
\begin{proof}
    We will construct the isomorphism \eqref{eq:BS-slide-1-copy}. The isomorphism \eqref{eq:BS-slide-2-copy} can be constructed analogously.
    Note that by \cref{prop:E-sts-splitting-1}, $\EBS_{1,2}  = E(T\backslash G_{1,2}/T)$ is an object in $\SBim_E(3)$. Note that $G_{1,2}$ is simply $\U(3)$ in our case.
    First we are going to show that 
    $\EBS_1 \circ R_2 \circ R_1(4)$  is equivalent to the fiber of the composite map
    \begin{equation}
        \EBS_{1,2} \xrightarrow{\mu_{121}} \EBS_{\underline{121}} \xrightarrow{\EBS_{\underline{12}} \otimes_{E(BT)} m_1} \EBS_{\underline{12}}.
    \end{equation}
    By \cref{lem:total-fiber-of-tensors}, we see that 
    $\EBS_1 \circ R_2 \circ R_1(4)$ is equivalent to the total fiber of 
    \begin{equation}\label{eq:condition-3-square-1}
        \begin{tikzcd}
            \EBS_{\underline{121}} \ar[rrrr, "\EBS_1 \otimes_{E(BT)} m_2 \otimes_{E(BT)} \EBS_1"] \ar[d, "\EBS_{\underline{12}} \otimes_{E(BT)} m_1"] &&&& \EBS_{\underline{11}} \ar[d, "\EBS_1 \otimes_{E(BT)} m_1"] \\
            \EBS_{\underline{12}} \ar[rrrr, "\EBS_1 \otimes_{E(BT)} m_2"'] &&&& \EBS_1.
        \end{tikzcd}
    \end{equation}
    By \cref{prop:E-BS_s-splitting-2}
    and \cref{cor:E-sts-splitting-2}, 
    we have compatible splittings  
    \begin{equation}
        \EBS_{\underline{11}} \simeq \EBS_1 \oplus \EBS_1(2), \quad
        \EBS_{\underline{121}} \simeq \EBS_{1,2} \oplus \EBS_1(2), 
    \end{equation}
    such that the top horizontal map of \eqref{eq:condition-3-square-1} is of the form 
    \begin{equation}
        \begin{pmatrix}
            ? & ? \\ 
            0 & \id
        \end{pmatrix}.
    \end{equation}
    Now we can write \eqref{eq:condition-3-square-1} as 
    \begin{equation}\label{eq:condition-3-square-2}
        \begin{tikzcd}
            \EBS_{1,2} \oplus \EBS_1(2) \ar[rrrr]
            \ar[d] &&&& \EBS_1 \oplus \EBS_1(2) \ar[d, "{(\id, 0)}"] \\
            \EBS_{\underline{12}} \ar[rrrr, "\EBS_1 \otimes_{E(BT)} m_2"'] &&&& \EBS_1.
        \end{tikzcd} 
    \end{equation}
By \cref{obs:eliminating-direct-sums}, the total fiber of \eqref{eq:condition-3-square-2} is equivalent to the total fiber of 
    \begin{equation}\label{eq:condition-3-square-3}
        \begin{tikzcd}
            \EBS_{1,2} \oplus \EBS_1(2) \ar[rrrr, "{(0, \id)}"]
            \ar[d] &&&& \EBS_1(2)  \ar[d] \\
            \EBS_{\underline{12}} \ar[rrrr] &&&& 0.
        \end{tikzcd} 
    \end{equation}
    By \cref{lem:total-fiber-inductive} and taking fibers horizontally, the total fiber of \eqref{eq:condition-3-square-3} is equivalent to the fiber of the composite map
    \begin{equation}\label{eq:sts}
        \EBS_{1,2} \xrightarrow{\mu_{121}} \EBS_{\underline{121}} \xrightarrow{\EBS_{\underline{12}} \otimes_{E(BT)} m_1} \EBS_{\underline{12}}.
    \end{equation}
    Applying the same argument, we can show that $R_2 \circ R_1 \circ \EBS_2(4)$ is equivalent to the fiber of the composite 
    \begin{equation}\label{eq:tst}
        \EBS_{1,2} \xrightarrow{\mu_{212}} \EBS_{\underline{jij}} 
        \xrightarrow{m_2 \otimes_{E(BT)} \EBS_{\underline{12}}} \EBS_{\underline{12}}.
    \end{equation}
    By \cref{obs:E-sts-commuting-square} (in particular \eqref{eq:sts-to-st-commuting}), the composites in \eqref{eq:sts} and \eqref{eq:tst} are equivalent. It follows that
    $\EBS_1 \circ R_2 \circ R_1$ is equivalent to  $R_2 \circ R_1 \circ \EBS_2$.
\end{proof}

Putting it altogether, we have our main theorem:
\begin{theorem}\label{thm:main}
   Let $E$ be a connective $\EE_\infty$-ring spectrum together with a complex $\EE_\infty$-orientation. Then there exists a contractible space of braidings on $\Kbloc(\SBim_E)$
   in $\Cat[\st_E]_{/\PolyMorc_E}$ whose braiding on $2$ strands is isomorphic to the Rouquier complex $R$ defined in \cref{def:rouquier-complex}.
\end{theorem}

\begin{proof}
    By \cref{prop:reducing-to-prebraid} and \cref{cor:reduce-to-Rouquier-and-slide}, it suffices to check conditions 
    (1), (2), (3) of \cref{cor:reduce-to-Rouquier-and-slide} for the Rouquier complex $R$. These conditions are respectively proven in \cref{prop:condition-1}, \cref{prop:condition-2}, and \cref{prop:condition-3}.
\end{proof}

We end with a remark about generalizing to the non-connective case:
\begin{remark}\label{rem:non-connective-E}
    We restrict ourselves to $E$ being a \emph{connective} $\EE_\infty$ spectrum because $\add_E$ is only well-defined if $E$ is connective. Without assuming connectivity, we can still define $\SBim_E$ and $\Kbloc(\SBim_E)$, as well as construct a braiding on $\Kbloc(\SBim_E)$.
However, $\SBim_E$ will be only enriched in $\add$, and $\Kbloc(\SBim_E)$ will be an $\EE_2$-algebra in $\Cat[\st]$ rather than in $\Cat[\st_E]$.
\end{remark}
\bibliographystyle{amsalpha}
\bibliography{bib.bib}

\end{document}